\documentclass[12pt, a4]{amsart}
\usepackage{amsmath,amssymb,cite,mathrsfs}

\newtheorem{thm}{Theorem}

\newtheorem{cor}{Corollary}
\newtheorem {conj}{Conjecture}

\theoremstyle{definition}
\newtheorem{defin}{Definition}

\theoremstyle{remark}
\newtheorem{rem}{Remark}

\numberwithin{equation}{section}
\begin{document}
\title[The theory of universality for zeta and $L$-functions]
{A survey on the theory of universality for zeta and $L$-functions}
\author{Kohji Matsumoto}
\date{}
\maketitle
\begin{abstract}
We survey the results and the methods in the theory of universality
for various zeta and $L$-functions, obtained in these forty years after
the first discovery of the universality for the Riemann zeta-function by
Voronin.
\end{abstract}

\tableofcontents

\section{Voronin's universality theorem}\label{sec-1}

Let $\mathbb{N}$ be the set of positive integers, 
$\mathbb{N}_0=\mathbb{N}\cup\{0\}$,
$\mathbb{Z}$ the ring of rational
integers, $\mathbb{Q}$ the field of rational numbers, $\mathbb{R}$ the field of
real numbers, and $\mathbb{C}$ the field of complex numbers.   In the present
article, the letter $p$ denotes a prime number.

For any open region $D\subset\mathbb{C}$, denote by $H(D)$ the space of
$\mathbb{C}$-valued holomorphic functions defined on $D$, equipped with the
topology of uniform convergence on compact sets.

For any subset $K\subset\mathbb{C}$, let $H^c(K)$ be the set of 
continuous functions defined on $K$, and are holomorphic in the interior of $K$,
and let $H_0^c(K)$ be the set of all elements of $H^c(K)$ which are
non-vanishing on $K$.
By ${\rm meas}\;S$ we mean the usual Lebesgue measure of the set $S$, and
by $\# S$ the cardinality of $S$.
 
Let $s=\sigma+it\in\mathbb{C}$ (where $i=\sqrt{-1}$), and
$\zeta(s)$ the Riemann zeta-function.   This function is defined by the
infinite series $\zeta(s)=\sum_{n=1}^{\infty}n^{-s}$ in the half-plane
$\sigma=\Re s>1$, and can be continued meromorphically to the whole of
$\mathbb{C}$.   It is well known that the investigation of
$\zeta(s)$ in the strip $0<\sigma<1$ is extremely important in number
theory, but its behaviour there is still quite mysterious.    A typical
example of expressing such mysterious feature of $\zeta(s)$ is Voronin's
{\it universality theorem}.

Consider the closed disc $K(r)$ with center $3/4$ and radius $r$, where
$0<r<1/4$.
Then Voronin \cite{Vor75} proved that for any $f\in H_0^c(K(r))$ and any
$\varepsilon>0$, there exists a positive number $\tau$ for which
\begin{align}\label{1-1}
\max_{s\in K(r)}|\zeta(s+i\tau)-f(s)|<\varepsilon
\end{align}
holds.   Roughly speaking, {\it any non-vanishing holomorphic function
can be approximated uniformly by a certain shift of $\zeta(s)$}. 
  
Actually, Voronin's proof essentially includes the fact that the set of
such $\tau$ has a positive lower density.    Moreover, now it is known
that the disc can be replaced by more general type of sets.   The modern
statement of Voronin's theorem is as follows.    
Let 
$$D(a,b)=\{s\in\mathbb{C}\;|\;a<\sigma<b\}.$$

\begin{thm}\label{thm1-1}
{\rm (Voronin's universality theorem)}
Let $K$ be a compact subset of $D(1/2,1)$ with connected complement, and 
$f\in H_0^c(K)$.   Then, for any $\varepsilon>0$,
\begin{align}\label{1-2} 
\liminf_{T\to\infty}\frac{1}{T}{\rm meas}\left\{\tau\in[0,T]\;\left|\;
\sup_{s\in K}|\zeta(s+i\tau)-f(s)|<\varepsilon\right.\right\}>0
\end{align}
holds.
\end{thm}

Let $\varphi(s)$ be a Dirichlet series, and let $K$ be a compact subset of 
$D(a,b)$ with connected complement.   
If
\begin{align}\label{1-3} 
\liminf_{T\to\infty}\frac{1}{T}{\rm meas}\left\{\tau\in[0,T]\;\left|\;
\sup_{s\in K}|\varphi(s+i\tau)-f(s)|<\varepsilon\right.\right\}>0
\end{align}
holds for any $f\in H_0^c(K)$ and any $\varepsilon>0$, 
then we call that the {\it universality} holds
for $\varphi(s)$ in the region $D(a,b)$.
Theorem \ref{thm1-1} implies that the universality
holds for the Riemann zeta-function in the region $D(1/2,1)$.

The Riemann zeta-function has the Euler product expression
$\zeta(s)=\prod_p(1-p^{-s})^{-1}$, where $p$ runs over all prime
numbers.    This is valid only in the region $\sigma>1$, but even
in the region $D(1/2,1)$, it is possible to show that
a finite truncation of 
the Euler product  ``approximates'' $\zeta(s)$ in a certain mean-value sense.
On the other hand, since 
$\{\log p\}_p$ is linearly independent over 
$\mathbb{Q}$, we can apply the Kronecker-Weyl approximation theorem to obtain
that any target function $f(s)$ can be approximated by the above finite
truncation.    This is the basic structure of
the proof of Theorem \ref{thm1-1}.

\begin{rem}\label{rem1-0}
Here we recall the statement of the Kronecker-Weyl theorem.   For $x\in\mathbb{R}$,
the symbol $||x||$ stands for the distance from $x$ to the nearest integer.
Let $\alpha_1,\ldots,\alpha_m$ be real numbers, linearly independent over
$\mathbb{Q}$.   Then, for any real numbers $\theta_1,\ldots,\theta_m$ and any
$\varepsilon>0$, 
\begin{align}\label{1-0}
\lim_{T\to\infty}\frac{1}{T}{\rm meas}\left\{\tau\in[0,T]\;\left|\;
||\tau\alpha_k-\theta_k||<\varepsilon\; (1\leq k\leq m)\right.\right\}>0
\end{align}
holds.
\end{rem}

So far, three proofs are known for this theorem.    Needless to say, one 
of them is Voronin's original proof, which is also reproduced in
\cite{KarVor92}.    This proof is based on, besides the above facts,
Pecherski{\u\i}'s rearrangement theorem \cite{Pec73} in Hilbert spaces.
The second proof is given by Good \cite{Goo81}, which will be discussed in
Section \ref{sec-16}.   Gonek was inspired by the idea of Good to write his
thesis \cite{Gon79}, in which he gave a modified version of Good's argument.
The third is a more probabilistic proof
due to Bagchi \cite{Bag81} \cite{Bag82}.    Bagchi \cite{Bag81} is an
unpublished thesis, but its contents are carefully expounded in
\cite{Lau96}.    A common feature of the work of Gonek and Bagchi is that they
both used the approximation theorem of Mergelyan \cite{Mer51}\cite{Mer52}.

\begin{rem}\label{rem1-01}
Mergelyan's theorem asserts that, when $K$ is  a compact subset of $\mathbb{C}$, 
any $f\in H^c(K)$ can be approximated by polynomials uniformly on $K$.    This is
a complex analogue of the classical approximation theorem of Weierstrass.
\end{rem}

Here we mention some pre-history.   In 1914, Bohr and Courant \cite{BohCou14} 
proved that, for any $\sigma$ satisfying $1/2<\sigma\leq 1$, the set
$$
\{\zeta(\sigma+i\tau)\;|\;\tau\in\mathbb{R}\}
$$
is dense in $\mathbb{C}$.
In the next year, 
Bohr \cite{Boh15} proved that the same result holds for $\log\zeta(\sigma+i\tau)$.
These results are called {\it denseness theorems}.

Before obtaining his universality theorem, Voronin \cite{Vor72} discovered the
following multi-dimensional analogue of the theorem of Bohr and Courant.   

\begin{thm}\label{thm1-2}
{\rm (Voronin \cite{Vor72})}
For any $\sigma$ satisfying $1/2<\sigma\leq 1$, the set
$$
\{(\zeta(\sigma+i\tau),\zeta^{\prime}(\sigma+i\tau),\ldots,
\zeta^{(m-1)}(\sigma+i\tau))\;|\;\tau\in\mathbb{R}\}
$$
{\rm (}where $\zeta^{(k)}$ denotes the $k$-th derivative{\rm)} is dense in
$\mathbb{C}^m$.
\end{thm}

\begin{rem}\label{rem1-1}
Actually Voronin proved a stronger result.   
For any $s=\sigma+it$ with $1/2<\sigma\leq 1$ and for
any $h>0$, the set
$$
\{(\zeta(s+inh),\zeta^{\prime}(s+inh),\ldots,                                   
\zeta^{(m-1)}(s+inh))\;|\;n\in\mathbb{N}\}                                          
$$
is dense in $\mathbb{C}^m$.
\end{rem}

The universality theorem of Voronin may be regarded as a natural next step, 
bacause it is a kind of infinite-dimensional analogue of the theorem of Bohr and
Courant, or a {\it denseness theorem in a function space}.

Another refinement of the denseness theorem of Bohr and Courant is the 
{\it limit theorem} on $\mathbb{C}$, due to Bohr and Jessen \cite{BohJes3032}.
Recently, this theorem is usually formulated by probabilistic terminology.
Let $\sigma>1/2$.   For any Borel subset $A\subset\mathbb{C}$, put
$$
P_{T,\sigma}(A)=\frac{1}{T}{\rm meas}\left\{\tau\in[0,T]\;|\;
\zeta(\sigma+i\tau)\in A\right\}.
$$
This $P_{T,\sigma}$ is a probability measure on $\mathbb{C}$.   Then the modern
formulation of the limit theorem of Bohr and Jessen is that there exists a
probability measure $P_{\sigma}$, to which $P_{T,\sigma}$ is convergent weakly
as $T\to\infty$ (see \cite[Chapter 4]{Lau96}).

Bagchi \cite{Bag81} proved an analogue of the above theorem of Bohr and Jessen
on a certain function space, and used it in his alternative proof of the
universality theorem.   Therefore, to prove some universality-type theorem by
Bagchi's method, it is necessary to obtain some {\it functional limit theorem}
similar to that of Bagchi.   There are indeed a lot of papers devoted to the
proofs of various functional limit theorems, for the purpose of showing various
universality theorems.   However in the following sections we do not mention
explicitly this closely related topic.   For the details of functional limit 
theorems, see Laurin{\v c}ikas \cite{Lau96} or J. Steuding \cite{Ste07}.   
The connection between the theory of
universality and functional limit theorems is also discussed in the author's
survey articles \cite{Mat04}\cite{Mat06}.

After the publication of Voronin's theorem in 1975, now almost forty
years have passed.    
Voronin's theorem attracted a lot of mathematicians, and hence, after Voronin,
quite many papers on universality theory have been published.
The aim of the present article is to survey the developments
in this theory in these forty years.
The developments in these years can be divided into three stages.

(I) The first stage: 1975 $\sim$ 1987.

(II) The second stage: 1996 $\sim$ 2007.

(III) The third stage: 2007 $\sim$ present.

\noindent In the next section we will give a brief discussion what were
the main topics in each of these stages.
\bigskip

The author expresses his sincere gratitude to Professors
R. Garunk{\v s}tis, R. Ka{\v c}inskait{\. e}, E. Karikovas, 
A. Laurin{\v c}ikas, R. Macaitien{\. e}, H. Mishou,
H. Nagoshi, T. Nakamura, {\L}. Pa{\' n}kowski, 
D. {\v S}iau{\v c}i{\= u}nas, and J. Steuding for their valuable comments
and/or sending relevant articles.

\section{A rough sketch of the history}\label{sec-2}

(I) The first stage.

This is the first decade after Voronin's paper.   The original 
impact of Voronin's discovery was still fresh.   Mathematicians who
were inspired by Voronin's paper tried to discuss various
generalizations, analogies, refinements and so on.    Here is the
list of main results obtained in this decade.

$\bullet$ Alternative proofs (mentioned in Section \ref{sec-1}).

$\bullet$ Generalizations to the case of Dirichlet $L$-functions,
Dedekind zeta-functions etc.

$\bullet$ The joint universality.

$\bullet$ The strong universality (for Hurwitz zeta-functions).

$\bullet$ The strong recurrence.

$\bullet$ The discrete universality.

$\bullet$ The $\chi$-universality.

$\bullet$ The hybrid universality.

$\bullet$ A quantitative result.

\noindent It is really amazing that many important aspects in 
universality theory, developed extensively in later decades, had
already been introduced and studied in this first decade.
Unfortunately, however, many of those results were written only
in the theses of Voronin \cite{Vor77}, Gonek \cite{Gon79} and
Bagchi \cite{Bag81}, all of them remain unpublished.    This 
situation is probably one of the reasons why in the next several years 
there were so few publications on universality.

(II) The second stage.

During several years around 1990, the number of publications
concerning universality is very small.    Of course it is not
completely empty.    For example, the book of Karatsuba and Voronin 
\cite{KarVor92} was published in this period.
But the author prefers
to choose the year 1996 as the starting point of the second stage,
because in this year the important book \cite{Lau96} of
Laurin{\v c}ikas was published.    This is the first textbook which
is mainly devoted to the theory of universality and related topics, and especially,
provides the details of unpublished work of Bagchi.    Thanks to
the existence of this book, many mathematicians of younger
generation can now easily go into the theory of universality.
In fact, in this second stage, a lot of students of 
Laurin{\v c}ikas started to publish their papers, and they formed
the strong Lithuanian school.

The main topic in this decade was probably the attempt to extend
the class of zeta and $L$-functions for which the universality
property holds.   It is now known that
the universality property is valid for a rather wide class of zeta-functions.
J. Steuding's lecture note \cite{Ste07} includes the exposition of
this result, and also of many other results obtained after the 
publication of the book of
Laurin{\v c}ikas \cite{Lau96}.    Therefore the publication year of 
this lecture note is appropriate to the end of the second stage. 

(III) The third stage.

Now comes the third, present stage.    The theory of universality is
now developing into several new directions.   The notions of

$\bullet$ the mixed universality,

$\bullet$ the composite universality,

$\bullet$ the ergodic universality,

\noindent were introduced recently.    Other topics in universality theory have
also been discussed extensively.

In the following sections, we will discuss more closely each topic in
universality theory.

\section{Generalization to zeta and $L$-functions with Euler products}\label{sec-3}

Is it possible to prove the universality property for other zeta and $L$-functions?
This is surely one of
the most fundamental question.   In Section \ref{sec-1} we explained that
a key point in the proof of Theorem \ref{thm1-1} is the Euler product
expression.    Therefore we can expect the universality property for
other zeta and $L$-functions which have Euler products.   

The universality of the following zeta and $L$-functions 
were proved in the first decade.

$\bullet$ Dirichlet $L$-functions 
$L(s,\chi)=\sum_{n=1}^{\infty}\chi(n)n^{-s}$
($\chi$ is a Dirichlet character; see Voronin \cite{Vor75})
in $D(1/2,1)$, 

$\bullet$ certain
Dirichlet series with multiplicative coefficients (see Reich \cite{Rei77},
Laurin{\v c}ikas \cite{Lau79}\cite{Lau79b}\cite{Lau82}\cite{Lau83}\cite{Lau84}),

$\bullet$ Dedekind
zeta-functions $\zeta_F(s)=\sum_{\mathfrak{a}}N(\mathfrak{a})^{-s}$, where
$F$ is a number field, $\mathfrak{a}$ denotes a non-zero integral ideal, and 
$N(\mathfrak{a})$ its norm (see Voronin \cite{Vor77}\cite{Vor79},
Gonek \cite{Gon79}, Reich \cite{Rei80}\cite{Rei82}). 
Here, the universality can be proved in 
$D(1/2,1)$ if $F$ is an Abelian extension of $\mathbb{Q}$ (Gonek \cite{Gon79}), 
but for general $F(\neq\mathbb{Q})$, 
the proof is valid only in the narrower region $D(1-d_F^{-1},1)$, where
$d_F=[F:\mathbb{Q}]$.    The reason is that the mean value estimate for $\zeta_F(s)$,
applicable to the proof of universality, is known at present only 
in $D(1-d_F^{-1},1)$.

Later, Laurin{\v c}ikas also obtained universality theorems for
Matsumoto zeta-functions\footnote{This notion was first introduced by the author
\cite{Mat90}, to which the limit theorem of Bohr and Jessen was generalized. 
See also Ka{\v c}inskait{\.e}'s survey article \cite{Kac12}.}
(\!\!\cite{Lau98} under a strong assumption), 
and for the zeta-function attached to Abelian groups 
(\!\!\cite{Lau01}\cite{Lau03}).
Laurin{\v c}ikas and {\v S}iau{\v c}i{\=u}nas \cite{LauSia06} proved the
universality for the periodic 
zeta-function $\zeta(s,\mathfrak{A})=\sum_{n=1}^{\infty}a_n n^{-s}$
(where $\mathfrak{A}=\{a_n\}_{n=1}^{\infty}$ is a multiplicative periodic sequence 
of complex numbers) when the technical condition
\begin{align}\label{3-0}
\sum_{m=1}^{\infty}|a_{p^m}|p^{-m/2}\leq c
\end{align}
(with a certain constant $c<1$) holds for any prime $p$.
Schwarz, R. Steuding and J. Steuding \cite{SchSteSte07} proved another universality 
theorem on certain general Euler products with conditions on the asymptotic
behaviour of coefficients.

However, there was an obstacle when we try to generalize further.
A typical class of $L$-functions with Euler products is that of 
automorphic $L$-functions.   Let $g(z)=\sum_{n=1}^{\infty}a(n)e^{2\pi inz}$ be 
a holomorphic normalized Hecke-eigen cusp form of weight $\kappa$
with respect to $SL(2,\mathbb{Z})$ 
and let $L(s,g)=\sum_{n=1}^{\infty}a(n)n^{-s}$ be the
associated $L$-function.    The universality of $L(s,g)$ was first
discussed by Ka{\v c}{\. e}nas and Laurin{\v c}ikas \cite{KacLau98}, but they
showed the universality only under a very strong assumption.

What was the obstacle?   The asymptotic formula
\begin{align}\label{3-1}
\sum_{p\leq x}\frac{1}{p}=\log\log x+a_1+O(\exp(-a_2\sqrt{\log x}))
\end{align}
(where $a_1$, $a_2$ are constants)
is classically well known, and is used in the proof of Theorem
\ref{thm1-1}.   However the corresponding asymptotic formula for
the sum $\sum_{p\leq x}|a(p)|p^{-1}$ is not known.
To avoid this obstacle, Laurin{\v c}ikas and the author \cite{LauMat01}
invented a new method of using \eqref{3-1}, combined with the known
asymptotic formula
\begin{align}\label{3-2}
\sum_{p\leq x}|\widetilde{a}(p)|^2=\frac{x}{\log x}(1+o(1)),
\end{align}
where $\widetilde{a}(p)=a(p)p^{-(\kappa-1)/2}$.
This method is called the {\it positive density method}.
Modifying Bagchi's argument by virtue of this positive density method, 
one can show the following unconditional result.

\begin{thm}\label{thm3-1}
{\rm (Laurin{\v c}ikas and Matsumoto \cite{LauMat01})}
The universality holds for $L(s,g)$ in the region
$D(\kappa/2,(\kappa+1)/2)$.
\end{thm}

The positive density
method was then applied to prove the universality for more general class of 
$L$-functions; certain Dirichlet series with multiplicative coefficients
(Laurin{\v c}ikas and {\v S}le{\v z}evi{\v c}ien{\.e} \cite{LauSle02}),
$L$-functions attached to new forms with respect to congruence subgroups 
(Laurin{\v c}ikas, Matsumoto and J. Steuding \cite{LauMatSte03}), 
$L$-functions attached to a cusp form with character
(Laurin{\v c}ikas and Macaitien{\.e} \cite{LauMac12}),
and a certain subclass of the Selberg class\footnote{The notion
of Selberg class was introduced by Selberg \cite{Sel92}.   For basic definitions and
results in this theory, consult a survey \cite{KaczPer99} of Kaczorowski and Perelli,
or J. Steuding \cite{Ste07}.}
(J. Steuding \cite{Ste03}).   
J. Steuding extended his result further
in his lecture note \cite{Ste07}.   He introduced a wide class $\widetilde{S}$ of
$L$-functions defined axiomatically and proved the universality for elements of
$\widetilde{S}$.   The class $\widetilde{S}$, now sometimes called the Steuding
class, is not included in the Selberg class, but is a subclass of the class of 
Matsumoto zeta-functions.

Since the Shimura-Taniyama conjecture has been established, we now know that the
$L$-function $L(s,E)$ attached to a non-singular elliptic curve $E$ over
$\mathbb{Q}$ is an $L$-function attached to a new form.   Therefore the universality
for $L(s,E)$ is included in \cite{LauMatSte03}.    
The universality of positive powers of $L(s,E)$ was studied in 
Garbaliauskien{\.e} and Laurin{\v c}ikas \cite{GarbLau05}.

Mishou \cite{Mis01}\cite{Mis03} used a variant of the positive density method to show
the universality for Hecke $L$-functions of algebraic number fields
in the region $D(1-d_F^{-1},1)$.   
Lee \cite{Lee12} showed that, under the assumption of a certain density
estimate of the number of zeros, it is possible to prove the universality for 
Hecke $L$-functions in the region $D(1/2,1)$.
The universality for Artin $L$-functions was proved by Bauer \cite{Bau03} by
a different method, based on Voronin's original idea.

Let $g_1$ and $g_2$ be cusp forms.   The universality for the Rankin-Selberg
$L$-function $L(s,g_1\otimes g_1)$ was shown by the author \cite{Mat01}, 
and for $L(s,g_1\otimes g_2)$ ($g_1\neq g_2$) was by Nagoshi \cite{Nag09} 
(both in the narrower region
$D(3/4,1)$).   The latter proof is based
on the above general result of J. Steuding \cite{Ste03}\cite{Ste07}.
The universality of symmetric $m$-th power $L$-functions ($m\leq 4$) and their
Rankin-Selberg $L$-functions was studied by Li and Wu \cite{LiWu07}.

Another general result obtained by the positive density method is the
following theorem, which is an extension of the result of J. Steuding \cite{Ste03}.

\begin{thm}\label{thm3-2}
{\rm (Nagoshi and J. Steuding \cite{NagSte10})}
Let $\varphi(s)=\sum_{n=1}^{\infty}a(n)n^{-s}$ be a Dirichlet
series belonging to the Selberg class.   Denote the degree of $\varphi(s)$ by 
$d_{\varphi}$, and put 
$\sigma_{\varphi}=\max\{1/2,1-d_{\varphi}^{-1}\}$.
Assume that
\begin{align}\label{3-3}                                                                
\sum_{p\leq x}|a(p)|^2=\frac{x}{\log x}(\lambda+o(1))
\end{align}
holds with a certain positive constant $\lambda$.  Then the universality holds 
for $\varphi(s)$ in the region $D(\sigma_{\varphi},1)$. 
\end{thm} 

\begin{rem}\label{rem3-1}
When $\varphi$ is the Riemann zeta-function, the formula \eqref{3-3} (with
$\lambda=1$) is nothing but the prime number theorem.   Therefore we may say that
the positive density method enables us to prove the universality
for zeta-functions with Euler products, provided an
asymptotic formula of the prime-number-theorem type is known.
\end{rem}

Let $h$ be a Hecke-eigen Maass form, and let $L(s,h)$ be the associated $L$-function.
The universality for $L(s,h)$ was proved by Nagoshi \cite{Nag05}\cite{Nag07}.
In the proof of Theorem \ref{thm3-1} in \cite{LauMat01}, Deligne's estimate
(Ramanujan's conjecture) $|\widetilde{a}(p)|\leq 2$ is essentially used.   
Assumptions of the same type
are required in Steuding's general result \cite{Ste03}\cite{Ste07} and also in
Theorem \ref{thm3-2}.    
Since Ramanujan's conjecture for Maass forms has not yet been proved, Nagoshi
in his first paper \cite{Nag05} assumed the validity of Ramanujan's conjecture
for $h$ to show the universality.   Then in the second paper \cite{Nag07} he
succeeded to remove this assumption, by invoking the asymptotic formula for
the fourth power mean of the coefficients due to M. Ram Murty.

\begin{thm}\label{thm3-3}
{\rm (Nagoshi \cite{Nag07})}
The universality holds for $L(s,h)$ in the region $D(1/2,1)$.
\end{thm}

Another important class of zeta-functions which have Euler products is the class of
Selberg zeta-functions.   In this case, instead of the prime-number-theorem type 
of results, the prime geodesic theorem plays an important role.
Let 
$$\mathcal{D}=\{d\in\mathbb{N}\;|\;d\equiv 0\; {\rm or}\; 1 \;{\rm(mod\; 4)},
d \;{\rm \;is\; not\; a\; square}\}.$$ 
For each $d\in\mathcal{D}$, let $h^+(d)$ be the number of 
inequivalent primitive quadratic forms of discriminant $d$,
and $\varepsilon(d)=(u(d)+v(d)\sqrt{d})/2$, where $(u(d),v(d))$ is the
fundamental solution of the Pell equation $u^2-v^2d=4$. 
Then, the prime geodesic theorem for $SL(2,\mathbb{Z})$ implies
\begin{align}\label{3-4}                                                                
\sum_{d\in\mathcal{D}\atop \varepsilon(d)^2\leq x}h^+(d)=\int_0^x\frac{dt}{\log t}
+O(x^{\alpha})
\end{align}
with a certain $\alpha<1$.   

\begin{thm}\label{thm3-4}
{\rm (Drungilas, Garunk{\v s}tis and Ka{\v c}{\.e}nas \cite{DruGarKac13})}
Let $Z(s)$ be the Selberg zeta-function attached to $SL(2,\mathbb{Z})$.
If \eqref{3-4} holds, then the universality holds for $Z(s)$ in the region
$D(1/2+\alpha/2,1)$.
\end{thm}

As for the value of $\alpha$, it is known that one can take
$\alpha=71/102+\varepsilon$ for any $\varepsilon>0$ (Cai \cite{Cai02}).
It is conjectured that one could take $\alpha=1/2+\varepsilon$.
If the conjecture is true, then Theorem \ref{thm3-4} implies that $Z(s)$
has the universality property in $D(3/4,1)$.   The paper \cite{DruGarKac13}
includes a discussion which suggests that $D(3/4,1)$ is the widest possible
region where the universality for $Z(s)$ holds.

\section{The joint universality for zeta and $L$-functions with Euler 
products}\label{sec-4}

The results presented in the previous sections give approximation
properties of some single zeta or $L$-function.    Here we discuss
simultaneous approximations by several zeta or $L$-functions.

Let $K_1,\ldots,K_r$ be compact subsets of $D(a,b)$ with connected
complements, and $f_j\in H_0^c(K_j)$ ($1\leq j\leq r$).
If Dirichlet series 
$\varphi_1(s),\ldots,\varphi_r(s)$ satisfy
\begin{align}\label{4-1} 
\liminf_{T\to\infty}\frac{1}{T}{\rm meas}\Bigl\{\tau\in[0,T]\;\left|\;
\sup_{s\in K_j}|\varphi_j(s+i\tau)-f_j(s)|<\varepsilon\right.\Bigr.\\
\Bigl.(1\leq j\leq r)\Bigr\}>0\notag
\end{align}
for any $\varepsilon>0$, we call the {\it joint universality} holds for
$\varphi_1(s),\ldots,\varphi_r(s)$ in the region $D(a,b)$.
The joint universality for Dirichlet $L$-functions was already obtained in the
first decade by Voronin \cite{Vor75b}\cite{Vor77}\cite{Vor79}, 
Gonek \cite{Gon79}, and
Bagchi \cite{Bag81}\cite{Bag82}, independently of each other:

\begin{thm}\label{thm4-1}
{\rm (Voronin, Gonek, Bagchi)}
Let $\chi_1,\ldots,\chi_r$ be pairwise non-equivalent Dirichlet
characters, and $L(s,\chi_1),\ldots,L(s,\chi_r)$ the
corresponding Dirichlet $L$-functions.    
Then the joint universality holds for $L(s,\chi_1),\ldots,L(s,\chi_r)$
in the region $D(1/2,1)$.
\end{thm}

To prove such a theorem of simultaneous approximations, it is obviously necessary
that the behaviour of $L$-functions appearing in the theorem should be
``independent'' of each other.   In the situation of Theorem \ref{thm4-1}, this
is embodied by the orthogonality relation of Dirichlet characters, which is
essentially used in the proof.

The joint universality for Dedekind zeta-functions was studied by Voronin
\cite{Vor77}\cite{Vor79}.   Bauer's work \cite{Bau03} mentioned in the preceding 
section actually proves a joint universality theorem on Artin $L$-functions,
in the region $D(1-(2d_F)^{-1},1)$.   Lee \cite{Lee12b} extended the region to
$D(1/2,1)$ under the assumption of a certain density estimate.

Let $g_j$ ($1\leq j\leq r$) be multiplicative arithmetic functions.
Laurin{\v c}ikas \cite{Lau86} considered the joint universality of the associated 
Dirichlet series
$\sum_{n=1}^{\infty}g_j(n)n^{-s}$ ($1\leq j\leq r$).   In this case, the
``independence'' condition is given by the following matrix condition.
Let $P_1,\ldots,P_k$ ($k\geq r$) be certain sets of prime numbers, 
with the condition that $\sum_{p\leq x,p\in P_i}p^{-1}$ satisfies a good
asymptotic formula, and assume that 
$g_j(n)$ is a constant $g_{jl}$ on the set $P_l$ ($1\leq j\leq r$, $1\leq l\leq k$).
Laurin{\v c}ikas \cite{Lau86} 
proved a joint universality theorem under the condition that
the rank of the matrix $(g_{jl})_{1\leq j\leq r,1\leq l\leq k}$ is equal to $r$.

Laurin{\v c}ikas frequently used various matrix conditions to obtain joint
universality theorems.   A joint universality theorem on Matsumoto
zeta-functions under a certain matrix condition was proved in \cite{Lau98b}.
A joint universality for automorphic $L$-functions under a certain matrix
condition was discussed in \cite{LauMat02}.

A matrix condition naturally appears in the joint universality theory of
periodic zeta-functions (see \cite{LauMacSia07}\cite{LauMac09}).
Let $\mathfrak{A}_j=\{a_{jn}\}_{n=1}^{\infty}$  be a multiplicative periodic 
sequence (whose least period we denote by $k_j$) of complex numbers, and 
$\zeta(s,\mathfrak{A}_j)$ the associated periodic zeta-function ($1\leq j\leq r$).
Let $k$ be the least common multiple of $k_1,\ldots,k_r$.   Define the matrix
$A=(a_{jl})_{j,l}$, where $1\leq j\leq r$ and $1\leq l\leq k$, $(l,k)=1$.
Then Laurin{\v c}ikas and Macaitien{\.e} \cite{LauMac09}
proved the joint universality for 
$\zeta(s,\mathfrak{A}_1),\ldots,\zeta(s,\mathfrak{A}_r)$ in the region
$D(1/2,1)$, if we assume ${\rm rank}(A)=r$ and a technical condition
similar to \eqref{3-0}.

Using the positive density method, it is possible to prove a joint universality
theorem for twisted automorphic $L$-functions.   Let
$g(z)=\sum_{n=1}^{\infty}a(n)e^{2\pi inz}$ be a holomorphic normalized
Hecke-eigen cusp form, 
$\chi_1,\ldots,\chi_r$ be pairwise non-equivalent Dirichlet
characters, and $L(s,g,\chi_j)=\sum_{n=1}^{\infty}a(n)\chi_j(n)n^{-s}$ the
associated $\chi_j$-twisted $L$-function.   

\begin{thm}\label{thm4-2}
{\rm (Laurin{\v c}ikas and Matsumoto \cite{LauMat04})}
The joint universality holds for
$L(s,g,\chi_1),\ldots,L(s,g,\chi_r)$ in the region $D(\kappa/2,(\kappa+1)/2)$.
\end{thm}

To prove this result, we need a prime number theorem for $a(p)$ in arithmetic 
progressions, that is
\begin{align}\label{4-2}
\sum_{p\leq x\atop p\equiv h ({\rm mod}\; q)}|\widetilde{a}(p)|^2
=\frac{1}{\varphi(q)}\frac{x}{\log x}(1+o(1)),
\end{align}
where $(h,q)=1$ and $\varphi(q)$ is Euler's totient function.

J. Steuding \cite[Theorem 12.8]{Ste07} generalized Theorem \ref{thm4-2}
to the Steuding class $\widetilde{S}$.
A joint version of \cite{LauSle02} was given by 
{\v S}le{\v z}evi{\v c}ien{\.e} \cite{Sle02}.
A joint universality theorem on $L$-functions of elliptic curves, under a certain
matrix condition, was given by Garbaliauskien{\.e}, Ka{\v c}inskait{\.e} and 
Laurin{\v c}ikas \cite{GarbKacLau04}.

Let $\varphi_j(s)=\sum_{n=1}^{\infty}a_j(n)n^{-s}$ ($j=1,2$) be elements of 
the Selberg class.   The following orthogonality conjecture of Selberg \cite{Sel92} 
is well known: if $\varphi_1(s)$, $\varphi_2(s)$ are primitive, then
\begin{align}\label{4-3}
\sum_{p\leq x}\frac{a_1(p)\overline{a_2(p)}}{p}=
\left\{
\begin{array}{ll}
\log\log x+O(1) & {\rm if}\quad \varphi_1=\varphi_2,\\
O(1)  &  {\rm otherwise}.
\end{array}
\right.
\end{align}  

Inspired by this conjecture, J. Steuding \cite[Section 12.5]{Ste07} proposed:

\begin{conj}\label{conj4-1}
{\rm (J. Steuding)}
Any finite collection of distinct primitive functions in the Selberg
class is jointly universal
\footnote{
J. Steuding also mentioned the expectation that any two functions 
$\varphi_1,\varphi_2$ in the Selberg class would be jointly
universal if and only if
$
\sum_{p\leq x}\frac{a_1(p)\overline{a_2(p)}}{p}=O(1).
$
However H. Nagoshi, and then H. Mishou, pointed out that there are counter examples
to this statement.}.
\end{conj}

Towards this conjecture, recent progress has been mainly due to Mishou.
In \cite{Mis12}, Mishou proved the following.   Consider two strips
$D_1=D(1/2,3/4)$ and $D_2=D(3/4,1)$.    Let $K_j$ be a compact subset of $D_j$
and $f_j\in H_0^c(K_j)$ ($j=1,2$).   Then
\begin{align}\label{4-5}
\liminf_{T\to\infty}\frac{1}{T}{\rm meas}\Bigl\{\tau\in[0,T]\;\left|\;                  
\sup_{s\in K}|\zeta(s+i\tau)-f_1(s)|<\varepsilon,\right.\Bigr.\\             
\Bigl.|L(s+i\tau,g)-f_2(s)|<\varepsilon\Bigr\}>0\notag 
\end{align}
holds, where $L(s,g)$ is the automorphic $L$-functions attached to a certain cusp
form $g$.    In Mishou \cite{Mis13}, this result was generalized to the case of
several $L$-functions belonging to the Selberg class.

The result \eqref{4-5} is weaker than the joint universality, because $K_1$ and
$K_2$ are in the strips disjoint to each other. 
In \cite{MisMZ}, Mishou succeeded in removing this restriction to obtain the
following theorem.
Let $g, g_1, g_2$ be holomorphic normalized Hecke-eigen cusp forms.

\begin{thm}\label{thm4-3}
{\rm (Mishou \cite{MisMZ})}
{\rm (i)} $\zeta(s)$ and $L(s,g)$ are jointly universal in $D(1/2,1)$.

{\rm (ii)} If $g_1$ and $g_2$ are distinct, then $L(s,g_1)$ and $L(s,g_2)$ are
jointly universal in $D(1/2,1)$.

{\rm (iii)} $\zeta(s)$ and $L(s,{\rm sym}^2 g)$ are jointly universal in
$D(2/3,1)$.

{\rm (iv)} If $g_1$ and $g_2$ are distinct, then $\zeta(s)$ and
$L(s,g_1\otimes g_2)$ are jointly universal in $D(3/4,1)$.

{\rm (v)} $L(s,g_1)$ and $L(s,g_1\otimes g_2)$ are 
jointly universal in $D(3/4,1)$.
\end{thm}

\begin{rem}\label{rem4-1}
The universality theorem for $L(s,g\otimes g)$ by the author \cite{Mat01}
(mentioned in Section \ref{sec-3}) was proved in $D(3/4,1)$, but the above theorem
of Mishou especially implies that the universality for $L(s,g\otimes g)$ is valid
in the wider region $D(2/3,1)$.
\end{rem}

A remarkable feature of Mishou's method is that it does not depend on any periodicity 
of coefficients.   His proof is based on orthogonality relations of Fourier
coefficients.   Theorem \ref{thm4-3}
is a strong support to Conjecture \ref{conj4-1}.   

\section{The strong universality}\label{sec-5}

So far we have talked about universality only for zeta and $L$-functions
with Euler products.    
However already in the first decade, the universality for 
zeta-functions without Euler products was also studied.    Let $0<\alpha\leq 1$.   
The Hurwitz zeta-function with the parameter $\alpha$ is defined by
$\zeta(s,\alpha)=\sum_{n=0}^{\infty}(n+\alpha)^{-s}$, and does not have
the Euler product (except for the special cases $\alpha=1, 1/2$).
The known universality theorem for $\zeta(s,\alpha)$ is as follows.

\begin{thm}\label{thm5-1}
{\rm (Bagchi \cite{Bag81}, Gonek \cite{Gon79})}
Let $K$ be a compact subset of $D(1/2,1)$ with connected complement, and 
$f\in H^c(K)$.   Then for any $\varepsilon>0$, 
\begin{align}\label{5-1} 
\liminf_{T\to\infty}\frac{1}{T}{\rm meas}\left\{\tau\in[0,T]\;\left|\;
\sup_{s\in K}|\zeta(s+i\tau,\alpha)-f(s)|<\varepsilon\right.\right\}>0
\end{align}
holds, provided $\alpha$ is transcendental or rational $(\neq 1,1/2)$.
\end{thm}

To prove this theorem, 
when $\alpha$ is transcendental, we use the fact that the elements of the set
$$
\{\log(n+\alpha)\;|\;n\in\mathbb{N}_0\}
$$
are linearly independent over
$\mathbb{Q}$.    On the other hand, when $\alpha=a/b$ is rational,
then in view of the formula
\begin{align}\label{5-2}
\zeta(s,a/b)=\frac{b^s}{\varphi(b)}\sum_{\chi({\rm mod}\; b)}
\overline{\chi}(a)L(s,\chi),
\end{align}
we can reduce the problem to the joint universality of Dirichlet $L$-functions,
so we can apply Theorem \ref{thm4-1}.
    The case of algebraic irrational $\alpha$ is still open.

A remarkable point is that, in the statement of Theorem \ref{thm5-1},
we do not assume that the target function $f(s)$ is non-vanishing on $K$.   
This is a big
difference from Theorem \ref{thm1-1}, and when a universality-type
theorem holds without the non-vanishing assumption, we call it a
{\it strong universality theorem}.

Strong universality has an important application to the theory of 
zero-distribution.   Let $a<\sigma_1<\sigma_2<b$, and let
\begin{align*}
N(t;\sigma_1,\sigma_2;\varphi)=
\#\{\rho\in\mathbb{C}\;|\;\sigma_1\leq\Re\rho\leq\sigma_2,
0\leq\Im\rho\leq T,\varphi(\rho)=0\}
\end{align*}
for a function $\varphi$.
(In the above definition, zeros are counted with
multiplicity.)
Then we have the following consequence.

\begin{thm}\label{thm5-2}
If $\varphi(s)$ is strongly universal in the region $D(a,b)$, then
there exists a positive constant $C$ for which
\begin{align}\label{5-3}                                                                
N(T;\sigma_1,\sigma_2;\varphi)\geq C T 
\end{align}
holds for any $\sigma_1,\sigma_2$ satisfying $a<\sigma_1<\sigma_2<b$.
\end{thm}

\begin{proof}
Let $\delta$ be a small positive number, $0<\varepsilon<\delta$, 
$\sigma_1<\sigma_0<\sigma_2$, and $K=\{s\in\mathbb{C}\;|\;|s-\sigma_0|<\delta\}$.   
We choose $\delta$ so small that $K\subset D(a,b)$.   
We apply the strong universality to this $K$, $f(s)=s-\sigma_0$ and $\varepsilon$ 
to obtain that the set of real numbers $\tau$ such that
\begin{align*}
\sup_{|s-\sigma_0|\leq\delta}|\varphi(s+i\tau)-f(s)|<\varepsilon
\end{align*}
is of positive lower density.   Then, for such $\tau$,
\begin{align*}
\sup_{|s-\sigma_0|\leq\delta}|\varphi(s+i\tau)-f(s)|<\delta=
\inf_{|s-\sigma_0|=\delta}|f(s)|.
\end{align*}
Therefore by Rouch{\'e}'s theorem we see that 
$f(s)+(\varphi(s+i\tau)-f(s))=\varphi(s+i\tau)$ has the same number of zeros as that 
of $f(s)$ in the region $|s-\sigma_0|<\delta$, but the latter is obviously 1. 
That is, for each $\tau$ in the above set,
$\varphi(s)$ has one zero in $|s-(\sigma_0+i\tau)|<\delta$.
\end{proof}

\begin{cor}\label{cor5-1}
If $\alpha$ is transcendental or rational $(\neq 1,1/2)$, then 
$$C_1 T\leq N(T;\sigma_1,\sigma_2;\zeta(s,\alpha))\leq C_2 T$$
holds for any $1/2<\sigma_1<\sigma_2<1$.
\end{cor}

As for the upper bound part of this corollary, 
see \cite[Chapter 8, Theorem 4.10]{LauGar02}.

Further topics on the application of universality to the distribution of zeros
will be discussed in Section \ref{sec-9} and Section \ref{sec-15}.

Now strong unversality theorems are known for many other zeta-functions.
The Estermann zeta-function is defined by
\begin{align}\label{5-4}
E\left(s,\frac{k}{l},\alpha\right)=\sum_{n=1}^{\infty}\sigma_{\alpha}(n)
\exp\left(2\pi i\frac{k}{l}n\right)n^{-s},
\end{align}
where $k$ and $l$ are coprime integers and 
$\sigma_{\alpha}(n)=\sum_{d|n}d^{\alpha}$.   
The strong universality for
$E(s,k/l,\alpha)$ was studied in Garunk{\v s}tis, Laurin{\v c}ikas, 
{\v S}le{\v z}evi{\v c}ien{\.e} and J. Steuding\cite{GarLauSleSte02}.
The method is to write $E(s,k/l,\alpha)$ as a linear combination of
$E(s,\chi,\alpha)=\sum_{n=1}^{\infty}\sigma_{\alpha}(n)\chi(n)n^{-s}$, and apply 
a joint universality theorem for $E(s,\chi,\alpha)$ which follows from 
{\v S}le{\v z}evi{\v c}ien{\.e} \cite{Sle02}.

The Lerch zeta-function is defined by 
$$
\zeta(s;\alpha,\lambda)=\sum_{n=0}^{\infty}e^{2\pi i\lambda n}(n+\alpha)^{-s},
$$
where $0<\alpha\leq 1$ and $\lambda$ is real.   When $\lambda$ is an integer, then
$\zeta(s;\alpha,\lambda)$ reduces to the Hurwitz zeta-function, so we may assume
$0<\lambda<1$.
The strong universality for $\zeta(s;\alpha,\lambda)$ was proved by
Laurin{\v c}ikas \cite{Lau97} when $\alpha$ is transcendental.
The case when $\alpha$ is rational was discussed by Laurin{\v c}ikas \cite{Lau98c}.  
See also the textbook \cite{LauGar02} of Laurin{\v c}ikas and Garunk{\v s}tis.

Let $\mathfrak{B}=\{b_n\}$ is a periodic sequence, not necessarily
multiplicative.
The universality for periodic zeta-functions
$\zeta(s,\mathfrak{B})=\sum_{n=1}^{\infty}b_n n^{-s}$ was first
studied by J. Steuding \cite{Ste03b} (see \cite[Chapter 11]{Ste07}).
Kaczorowski \cite{Kacz09} proved that there exists a constant
$c_0=c_0(\mathfrak{B})$ such that the universality holds for 
$\zeta(s,\mathfrak{B})$, provided that
$$
\max_{s\in K}\Im(s)-\min_{s\in K}\Im(s)\leq c_0.
$$
This result is a consequence of the hybrid joint universality theorem of
Kaczorowski and Kulas \cite{KacKul07} (see Section \ref{sec-15}). 
Javtokas and Laurin{\v c}ikas \cite{JavLau06}\cite{JavLau06b} studied the strong
universality for periodic Hurwitz zeta-function
$$
\zeta(s,\alpha,\mathfrak{B})=\sum_{n=0}^{\infty}b_n(n+\alpha)^{-s}.
$$
They proved that the strong universality holds for $\zeta(s,\alpha,\mathfrak{B})$,
when $\alpha$ is transcendental.

A more general situation was considered by Laurin{\v c}ikas, Schwarz and
J. Steuding \cite{LauSchSte03}.   Let $\{\lambda_n\}_{n=1}^{\infty}$ be an
increasing sequence of real numbers, linearly independent over $\mathbb{Q}$, 
and $\lambda_n\to\infty$ as $n\to\infty$.   Define the general Dirichlet series
$f(s)=\sum_{n=1}^{\infty}a_n \exp(-\lambda_n s)$,
which is assumed to be convergent absolutely in the region $\sigma>\sigma_a$.
Put $r(x)=\sum_{\lambda_n\leq x}1$ and $c_n=a_n \exp(-\lambda_n\sigma_a)$.   
We suppose

(i) $f(s)$ cannot be represented as an Euler product,

(ii) $f(s)$ can be continued meromorphically to $\sigma>\sigma_1$, and holomorphic
in $D(\sigma_1,\sigma_a)$,

(iii) For $\sigma>\sigma_1$ it holds that $f(s)=O(|t|^{\alpha})$ with some 
$\alpha>0$,

(iv) For $\sigma>\sigma_1$ it holds that
$$
\int_{-T}^T |f(\sigma+it)|^2 dt=O(T),
$$

(v) $r(x)=Cx^{\kappa}+O(1)$ with a $\kappa>1$,

(vi) $|c_n|$ is bounded and 
$\sum_{\lambda_n\leq x}|c_n|^2=\theta r(x)(1+o(1))$ with a $\theta>0$.

Then we have

\begin{thm}\label{thm5-3}
{\rm (Laurin{\v c}ikas, Schwarz and J. Steuding \cite{LauSchSte03})}
If $f(s)$ satisfies all the above conditions, then the strong universality holds
for $f(s)$ in the region $D(\sigma_1,\sigma_a)$.
\end{thm}

In Section \ref{sec-3} and Section \ref{sec-5}, we have seen a lot of examples of
zeta and $L$-functions, with or without Euler products, for which the universality
property holds.   How general is this property expected to hold?   The following 
conjecture predicts that any ``reasonable'' Dirichlet series would satisfy the
universality property.

\begin{conj}\label{conj5-1}
{\rm (Yu. V. Linnik and I. A. Ibragimov)}
All functions given by Dirichlet series and meromorphically continuable to the left
of the half-plane of absolute convergence are universal in some suitable region.
\end{conj}

\begin{rem}\label{rem5-1}
Actually this conjecture has trivial counter-examples.   For example, let
$a_n=1$ if $n$ is a power of 2 and $a_n=0$ otherwise.   The series 
$\sum_{n=1}^{\infty}a_n n^{-s}$ can be continued to $(2^s-1)^{-1}$, which is
obviously not universal.   Therefore some additional condition should be
added to make the rigorous statement of the above conjecture.
\end{rem}

\section{The joint strong universality}\label{sec-6}

The joint universality property is also possible to be valid among zeta-functions
without Euler products.   The first attempt to this direction is a series of
papers of Laurin{\v c}ikas and the author 
\cite{LauMat00}\cite{LauMat06}\cite{LauMat07} on the joint universality for Lerch
zeta-functions.   Here, a matrix condition again appears.
Let $\lambda_1,\ldots,\lambda_r$ be rational numbers.   Write
$\lambda_j=a_j/q_j$, $(a_j,q_j)=1$, and let $k$ be the least
common multiple of $q_1,\ldots,q_r$.   Define the matrix 
$L=(\exp(2l\pi i\lambda_j))_{1\leq l\leq k,1\leq j\leq r}$.   Then,
by virtue of a variant of the positive density method, we have

\begin{thm}\label{thm6-1}
{\rm (Laurin{\v c}ikas and Matsumoto \cite{LauMat00} \cite{LauMat06})}
Suppose that $\alpha_1,\ldots,\alpha_r$ are algebraically independent over
$\mathbb{Q}$, and that ${\rm rank}(L)=r$.   Then the joint strong universality
holds for $\zeta(s,\alpha_1,\lambda_1),\ldots,\zeta(s,\alpha_r,\lambda_r)$
in the region $D(1/2,1)$\footnote{
In \cite{LauMat00}, the theorem is stated under the 
weaker assumption that $\alpha_1,\ldots,\alpha_r$ are transcendental, but 
(as is pointed
out in \cite{LauMat06}) this assumption should be replaced by the algebraic
independence over $\mathbb{Q}$.}
.
\end{thm}

The Lerch zeta-function $\zeta(s,\alpha,\lambda)$ with rational $\lambda$ is a
special case of periodic Hurwitz zeta-functions.   The joint strong universality
of periodic Hurwitz zeta-functions was first studied by Laurin{\v c}ikas 
\cite{Lau06}\cite{Lau07}.    Let $\mathfrak{B}_j=\{a_{nj}\}_{n=0}^{\infty}$
be periodic sequences with period $k_j$, $k$ be the least common multiple of
$k_1,\ldots,k_r$, and define $B=(a_{jl})_{1\leq j\leq r,1\leq l\leq k}$.

\begin{thm}\label{thm6-2}
{\rm (Laurin{\v c}ikas \cite{Lau07})}
If $\alpha$ is transcendental and ${\rm rank}(B)=r$, then the joint strong
universality holds for 
$\zeta(s,\alpha,\mathfrak{B}_1),\ldots,\zeta(s,\alpha,\mathfrak{B}_r)$
in the region $D(1/2,1)$.
\end{thm}
 
Next in \cite{JavLau08}\cite{LauSke09}, the joint universality for
$\zeta(s,\alpha_j,\mathfrak{B}_j)$ ($1\leq j\leq r$) was discussed.   Some
matrix conditions were still assumed in \cite{JavLau08}, but finally
in \cite{LauSke09}, a joint universality theorem free from any matrix condition
was obtained.    

\begin{thm}\label{thm6-3}
{\rm (Laurin{\v c}ikas and Skerstonait{\.e} \cite{LauSke09})}
Assume that the elements of the set
\begin{align}\label{6-1}                                                                
\{\log(n+\alpha_j)\;|\;1\leq j\leq r,\; n\in\mathbb{N}_0\}                                  
\end{align}
are linearly independent over $\mathbb{Q}$.    Then the joint strong universality
holds for
$\zeta(s,\alpha_1,\mathfrak{B}_1),\ldots,\zeta(s,\alpha_r,\mathfrak{B}_r)$
in the region $D(1/2,1)$.
\end{thm}

In the case $\mathfrak{B}_j=\{1\}_{n=0}^{\infty}$ for $1\leq j\leq r$ (that is,
the case of Hurwitz zeta-functions), the above result was already given in
Laurin{\v c}ikas \cite{Lau08b}.

In \cite{Lau08}\cite{LauSke09b}, a more general joint strong universality for
$\zeta(s,\alpha_j,\mathfrak{B}_{jl})$ ($1\leq j\leq r$, $1\leq l\leq l_j$ with
$l_j\in\mathbb{N}$) was discussed under certain matrix conditions.

Laurin{\v c}ikas \cite{Lau03b}\cite{Lau05} and \cite{GenLau04} (with Genys)
studied the joint strong universality for general Dirichlet series, under the same
assumptions as in Theorem \ref{thm5-3} and a certain matrix condition.

Now return to the problem of the joint universality for Lerch zeta-functions.
Theorem \ref{thm6-3} implies, especially, that the assumptions of 
Theorem \ref{thm6-1} can now be replaced
by just the linear independence of \eqref{6-1}.

Is the assumption \eqref{6-1} indeed weaker than the assumptions of Theorem 
\ref{thm6-1}?   The answer is yes, and the following result of Mishou \cite{Mis11b}
gives an example: Let $\alpha_1$, $\alpha_2$ be two transcendental numbers,
$0<\alpha_1,\alpha_2<1$, $\alpha_1\neq\alpha_2$, and 
$\alpha_2\in{\mathbb Q}(\alpha_1)$.   Then Mishou \cite{Mis11b} proved that
the joint strong universality holds for $\zeta(s,\alpha_1)$ and $\zeta(s,\alpha_2)$.
Dubickas \cite{Dub12} extended Mishou's result to the case of $r$ transcendental
numbers, which is also an extension of \cite{Lau08b}.

Let $m_1,m_2,m_3$ be relatively prime positive integers ($\geq 2$), and
$\lambda_0=n_3/m_3$ (with another integer $n_3$).   Nakamura \cite{Nak07} proved
the joint strong universality for
\begin{align}\label{6-2}
\left\{\left.\zeta\left(s,\alpha,\lambda_0+\frac{n_1}{m_1}+\frac{n_2}{m_2}\right)
\;\right|\;0\leq n_1<m_1,0\leq n_2<m_2\right\}
\end{align} 
when $\alpha$ is transcendental.   He pointed out that various other types of 
joint universality can be deduced from the above.

In \cite{LauMat07} and in Laurin{\v c}ikas \cite{Lau10}, the joint universality
of $\zeta(s,\alpha_j,\lambda_{j\mu_j})$ ($1\leq j\leq r$, $1\leq \mu_j\leq m_j$,
where $m_j$ is some positive integer) is discussed.
Write $\lambda_{j\mu_j}=a_{j\mu_j}/q_{j\mu_j}$, $(a_{j\mu_j},q_{j\mu_j})=1$, and
let $k_j$ be the least common multiple of $q_{j\mu_j}$ ($1\leq \mu_j\leq m_j$).   
Define
$L_j=(\exp(2l\pi i\lambda_{j\mu_j})_{1\leq l\leq k_j,1\leq \mu_j\leq m_j})$.
Then in \cite{Lau10} it is shown that if the elements of the set \eqref{6-1}
are linearly independent over $\mathbb{Q}$, and ${\rm rank}(L_j)=k_j$
($1\leq j\leq r$), then the joint strong universality holds for
$\zeta(s,\alpha_j,\lambda_{j\mu_j})$.

How about the joint universality for Lerch zeta-functions when the parameter
$\lambda$ is not rational?    
Nakamura \cite{Nak07} noted that the joint strong universality for \eqref{6-2}
also holds if we replace $\lambda_0$ by any non-rational real number.
Also, Nakamura \cite{Nak07b} extended
the idea in \cite{LauMat06} to obtain the following more general result.

\begin{thm}\label{thm6-4}
{\rm (Nakamura \cite{Nak07b})}
If $\alpha_1,\ldots,\alpha_r$ are algebraically independent over $\mathbb{Q}$, 
then for any real numbers $\lambda_1,\ldots,\lambda_r$ the joint strong universality
holds for $\zeta(s,\alpha_1,\lambda_1),\ldots,\zeta(s,\alpha_r,\lambda_r)$
in the region $D(1/2,1)$.
\end{thm}

On the other hand, Mishou \cite{MisJMSJ} proved 
the following measure-theoretic result.

\begin{thm}\label{thm6-5}
{\rm (Mishou \cite{MisJMSJ})}
There exists a subset $\Lambda\subset [0,1)^r$ whose $r$-dimensional Lebesgue
measure is $1$, and for any transcendental real number $\alpha$ and
$(\lambda_1,\ldots,\lambda_r)\in\Lambda$, the joint strong universality holds
for $\zeta(s,\alpha,\lambda_1),\ldots,\zeta(s,\alpha,\lambda_r)$ in the region
$D(1/2,1)$.
\end{thm}

Moreover in the same paper Mishou gives the following two explicit descriptions of 
$\lambda_1,\ldots,\lambda_r$ ($0\leq\lambda_j<1$) for which the above joint 
universality holds;

(i) $\lambda_1,\ldots,\lambda_r$ are algebraic irrational and
$1,\lambda_1,\ldots,\lambda_r$ are linearly independent over $\mathbb{Q}$,

(ii) $\lambda_1=\exp(u_1),\ldots,\lambda_r=\exp(u_r)$ where $u_1,\ldots,u_r$ are
distinct rational numbers.

Mishou's proof is based on two classical discrepancy estimates due to
W. M. Schmidt and H. Niederreiter.    These results lead Mishou to propose the
following conjecture.

\begin{conj}\label{conj6-1}
{\rm (Mishou \cite{MisJMSJ})}
The joint strong universality holds
for $\zeta(s,\alpha,\lambda_1),\ldots,\zeta(s,\alpha,\lambda_r)$ in the region
$D(1/2,1)$, for any transcendental real number $\alpha$ $(0<\alpha<1)$ and any
distinct real numbers $\lambda_1,\ldots,\lambda_r$ $(0\leq\lambda_j<1)$.
\end{conj}

For single zeta or $L$-functions, there is Conjecture \ref{conj5-1}, which asserts
that universality would hold for any ``reasonable'' Dirichlet series.   
As for the joint universality, the situation is much more complicated.
If there is some relation among several Dirichlet series, then the behaviour of those
Dirichlet series cannot be independent of each other, so the joint universality
among them cannot be expected.   Nakamura \cite{Nak07b} pointed out that
some collections
of Lerch zeta-functions cannot be jointly universal, because of the inversion
formula among Lerch zeta-functions.
In the same paper Nakamura introduced the generalized Lerch zeta-function of
the form
\begin{align}\label{6-3}
\zeta(s,\alpha,\beta,\gamma,\lambda)=\sum_{n=0}^{\infty}\frac{e^{2\pi i\lambda n}}
{(n+\alpha)^{s-\gamma}(n+\beta)^{\gamma}},
\end{align}
and showed that, under suitable choices of parameters the joint strong
universality sometimes holds, and sometimes does not hold.

In \cite{Nak08}, Nakamura considered more general series
\begin{align}\label{6-4}
\zeta(s,\alpha,\mathcal{C})=\sum_{n=0}^{\infty}\frac{c(n)}{(n+\alpha)^s},
\end{align}  
where $\mathcal{C}=\{c(n)\}_{n=0}^{\infty}$ is a bounded sequence of complex
numbers, and using it, constructed some counter-examples to the joint universality.
The proof in \cite{Nak08} is based on a non-denseness property and a limit
theorem on a certain function space.

The above results of Nakamura suggest that it is not easy to find a suitable
joint version of Conjecture \ref{conj5-1}.

\begin{rem}\label{rem6-1}
It is to be noted that there is the following simple principle of producing the
joint universality.   Let $\varphi\in H(D(\sigma_1,\sigma_2))$, and 
assume that $\varphi$ is universal.   Let
$K$ be a compact subset of $D(\sigma_1,\sigma_2)$ with connected complement,
$\lambda_1,\ldots,\lambda_r$ be complex numbers, and
$K_j=\{s+\lambda_j\;|\;s\in K\}$ ($1\leq j\leq r$).   Assume these $K_j$'s are
disjoint.   Then $\varphi_j(s)=\varphi(s+\lambda_j)$ ($1\leq j\leq r$) are
jointly universal.   If $\varphi$ is strongly universal, then $\varphi_j$'s are
jointly strongly universal.   This is the {\it shifts universality principle}
of Kaczorowski, Laurin{\v c}ikas and J. Steuding \cite{KaczLauSte06}. 
\end{rem}

\section{The universality for multiple zeta-functions}\label{sec-7}

An important generalization of the notion of zeta-functions is multiple 
zeta-functions, defined by certain multiple sums.   The history of the theory of
multiple zeta-functions goes back to the days of Euler, but extensive studies
started only in 1990s.

The problem of searching for universality theorems on multiple zeta-functions
was first proposed by the author \cite{Mat02}.   In this paper the author wrote
that one accessible problem would be the universality of Barnes multiple
zeta-functions    
\begin{align*}
\sum_{n_1=0}^{\infty}\cdots\sum_{n_r=0}^{\infty}
(w_1 n_1 +\cdots+w_r n_r+\alpha)^{-s}
\end{align*}
(where $w_1,\ldots,w_r,\alpha$ are parameters).   Nakamura \cite{Nak07b} pointed
out that his $\zeta(s,\alpha,\beta,\gamma,\lambda)$ (see \eqref{6-3}) includes 
the twisted Barnes double zeta-function of the form
\begin{align*}
\sum_{n_1=0}^{\infty}\sum_{n_2=0}^{\infty}\frac{e^{2\pi i\lambda(n_1+n_2)}}
{(n_1+n_2+\alpha)^s}
\end{align*}
as a special case.   Therefore \cite{Nak07b} includes a study of universality for
Barnes double zeta-functions.   More generally, the series 
$\zeta(s,\alpha,\mathcal{C})$ (see \eqref{6-4}) studied
in Nakamura \cite{Nak08} includes twisted Barnes $r$-ple zeta-functions (for any $r$).

The Euler-Zagier $r$-ple sum is defined by
\begin{align}\label{7-1}
\sum_{n_1>n_2>\cdots n_r\geq 1}n_1^{-s_1}n_2^{-s_2}
\cdots n_r^{-s_r},
\end{align}  
where $s_1,\ldots,s_r$ are complex variables.   
Nakamura \cite{Nak09} considered the universality of the following
generalization of \eqref{7-1} of Hurwitz-type:
\begin{align}\label{7-2}                                                                
\lefteqn{\zeta_r(s_1,\ldots,s_r;\alpha_1,\ldots,\alpha_r)=}\\
&\sum_{n_1>n_2>\cdots n_r\geq 0}(n_1+\alpha_1)^{-s_1}(n_2+\alpha_2)^{-s_2}
\cdots (n_r+\alpha_r)^{-s_r},\notag
\end{align}
where $0<\alpha_j\leq 1$ ($1\leq j\leq r$).   Nakamura's results suggest that
universality for the multiple zeta-function \eqref{7-2} is connected with the
zero-free region.   One of his main results is as follows.

\begin{thm}\label{thm7-1}
{\rm (Nakamura \cite{Nak09})}
Let $\Re s_2>3/2$, $\Re s_j\geq 1$ $(3\leq j\leq r)$.   Assume $\alpha_1$ is
transcendental, and 
$\zeta_{r-1}(s_2,\ldots,s_r;\alpha_2,\ldots,\alpha_r)\neq 0$.   Then the strong
universality holds for $\zeta_r(s_1,\ldots,s_r;\alpha_1,\ldots,\alpha_r)$ as a 
function in $s_1$ in the region $D(1/2,1)$.
\end{thm}

In \cite{Nak11}, Nakamura considered a generalization of Tornheim's double sum
of Hurwitz-type and proved the strong universality for it.

\section{The mixed universality}\label{sec-8}

In Section \ref{sec-4} we discussed the joint universality among zeta or $L$-functions
with Euler products.    Then in Section \ref{sec-6} we considered the joint universality
for those without Euler products.    Is it possible to combine these two directions 
to obtain certain joint universality results between two (or more) zeta-functions,
one of which has Euler products and the other does not?    The first affirmative
answers are due to Sander and J. Steuding \cite{SanSte06}, and to Mishou \cite{Mis07}.
The work of Sander and J. Steuding will be discussed later in Section \ref{sec-15}.
Here we state Mishou's theorem.

\begin{thm}\label{thm8-1}
{\rm (Mishou \cite{Mis07})}
Let $K_1,K_2$ be compact subsets of $D(1/2,1)$ with connected complements,
and $f_1\in H_0^c(K_1)$, $f_2\in H^c(K_2)$.
Then, for any $\varepsilon>0$, we have
\begin{align}\label{8-1}
\liminf_{T\to\infty}\frac{1}{T}{\rm meas}\left\{\tau\in[0,T]\;\left|\;
\sup_{s\in K_1}|\zeta(s+i\tau)-f_1(s)|<\varepsilon,\right.\right.\\
\left.\sup_{s\in K_2}|\zeta(s+i\tau,\alpha)-f_2(s)|<\varepsilon\right\}>0,\notag
\end{align}
provided $\alpha$ is transcendental.
\end{thm}

This type of universality is now called the {\it mixed universality}.
The essential point of the proof of this theorem is the fact that the elements
of the set
$$
\{\log(n+\alpha)\;|\;n\in\mathbb{N}_0\}\cup\{\log p\;|\;p:{\rm prime}\}
$$
are linearly independent over $\mathbb{Q}$.

Mishou's theorem was generalized to the periodic case by Ka{\v c}inskait{\.e}
and Laurin{\v c}ikas \cite{KacLau11}\footnote{This paper was published in 2011, but   
was already completed in 2007.}.
Let $\mathfrak{A}$ be a multiplicative periodic sequence satisfying \eqref{3-0}
and $\mathfrak{B}$ a (not necessarily multiplicative) periodic sequence.   
They proved that if $\alpha$ is transcendental, then the mixed universality holds for
$\zeta(s,\mathfrak{A})$ and $\zeta(s,\alpha,\mathfrak{B})$
in the region $D(1/2,1)$.

Laurin{\v c}ikas \cite{Lau10b} proved a further generalization.   Let
$\mathfrak{A}_1,\ldots,\mathfrak{A}_{r_1}$ be multiplicative periodic
sequences (with inequalities similar to \eqref{3-0}) and 
$\mathfrak{B}_1,\ldots,\mathfrak{B}_{r_2}$ be periodic
sequences.   Then, under certain matrix conditions, the mixed universality for
$\zeta(s,\mathfrak{A}_{j_1})$ ($1\leq j_1\leq r_1$) and 
$\zeta(s,\alpha_{j_2},\mathfrak{B}_{j_2})$ ($1\leq j_2\leq r_2$) holds, 
provided $\alpha_1,\ldots,\alpha_{r_2}$ are algebraically independent over
$\mathbb{Q}$.

Now mixed universality theorems are known for many pairs of zeta or $L$-functions.

$\bullet$ The Riemann zeta-function and several periodic Hurwitz zeta-functions
(Genys, Macaitien{\.e}, Ra{\v c}kauskien{\.e} and {\v S}iau{\v c}i{\=u}nas
\cite{GenMacRacSia10}),

$\bullet$  The Riemann zeta-function and several Lerch zeta-functions
(Laurin{\v c}ikas and Macaitien{\.e} \cite{LauMac13}),

$\bullet$  Several Dirichlet $L$-functions and several Hurwitz zeta-functions
(Janulis and Laurin{\v c}ikas \cite{JanLau13}),

$\bullet$ Several Dirichlet $L$-functions and several periodic Hurwitz 
zeta-functions (Janulis, Laurin{\v c}ikas, Macaitien{\.e} and
{\v S}iau{\v c}i{\=u}nas \cite{JanLauMacSia12}),

$\bullet$ An automorphic $L$-function and several periodic Hurwitz zeta-functions
(Laurin{\v c}ikas, Macaitien{\.e} and
{\v S}iau{\v c}i{\=u}nas \cite{LauMacSia11}, Macaitien{\.e} \cite{Mac12},
Pocevi{\v c}ien{\.e} and {\v S}iau{\v c}i{\=u}nas \cite{PocSia}, 
Laurin{\v c}ikas and {\v S}iau{\v c}i{\=u}nas \cite{LauSia12}).

\section{The strong recurrence}\label{sec-9}

In Section \ref{sec-5} we mentioned that the strong universality implies the
existence of many zeros in the region where universality is valid 
(Theorem \ref{thm5-2}).   This immediately gives the following corollary:

\begin{cor}\label{cor9-1}
The Riemann zeta-function $\zeta(s)$ cannot be strongly universal
in the region $D(1/2,1)$.
\end{cor}

Because if $\zeta(s)$ is strongly universal, then by Theorem \ref{thm5-2} we
have $N(T;\sigma_1,\sigma_2;\zeta)\geq CT$ for $1/2<\sigma_1<\sigma_2<1$,
which contradicts with the known zero-density estimate
$N(T;\sigma_1,\sigma_2;\zeta)=o(T)$. 

The same conclusion can be shown for many other zeta or $L$-functions, for which
some suitable zero-density estimate is known; or, under the assumption of the
analogue of the Riemann hypothesis.

On the other hand, if the Riemann hypothesis is true, then $\zeta(s)$ has no zero
in the region $D(1/2,1)$.    Therefore we can choose $f(s)=\zeta(s)$ in
Theorem \ref{thm1-1} to obtain
\begin{align}\label{9-1}
\liminf_{T\to\infty}\frac{1}{T}{\rm meas}\left\{\tau\in[0,T]\;\left|\;
\sup_{s\in K}|\zeta(s+i\tau)-\zeta(s)|<\varepsilon\right.\right\}>0.
\end{align}

This is called the {\it strong recurrence} property of $\zeta(s)$.
Bagchi discovered that the converse implication is also true.

\begin{thm}\label{thm9-1}
{\rm (Bagchi \cite{Bag81})}
The Riemann hypothesis for $\zeta(s)$ is true if and only if \eqref{9-1} holds
in the region $D(1/2,1)$.
\end{thm}

Bagchi himself extended this result to the case of Dirichlet $L$-functions in
\cite{Bag82}\cite{Bag87}.   The same type of result in terms of Beurling
zeta-functions was given by R. Steuding \cite{RSte06}.

\begin{rem}\label{rem9-1}
It is obvious that the notion of the strong recurrence is closely connected with
the notion of almost periodicity.   Bohr \cite{Boh22} proved that the 
Riemann hypothesis for $L(s,\chi)$ with a non-principal character $\chi$ is
equivalent to the almost periodicity of $L(s,\chi)$ in the region $\Re s>1/2$.
Recently Mauclaire \cite{Mau07} \cite{Mau09} studied the universality in a
general framework from the 
viewpoint of almost periodicity.
\end{rem}

Nakamura \cite{Nak09b} proved that, if $d_1=1,d_2,\ldots,d_r$ are algebraic real
numbers which are linearly independent over $\mathbb{Q}$, then the joint
universality of the form
\begin{align}\label{9-2}                                                                
\liminf_{T\to\infty}\frac{1}{T}{\rm meas}\left\{\tau\in[0,T]\;\left|\;                  
\sup_{s\in K_j}|\zeta(s+id_j\tau)-f_j(s)|<\varepsilon\right.\right.\\
 \Biggl.(1\leq j\leq r)\Biggr\}>0\notag
\end{align}
holds, where $K_j$ are compact subsets of $D(1/2,1)$ with connected complement
and $f_j\in H_0^c(K_j)$.   A key point of Nakamura's proof is the fact that the
elements of the set 
$\{\log p^{d_j}\;|\;p:{\rm prime},1\leq j\leq r\}$ are linearly independent over
$\mathbb{Q}$, which follows from Baker's theorem in transcendental number theory. 
From the above result it is immediate that
\begin{align}\label{9-3}                                                                
\liminf_{T\to\infty}\frac{1}{T}{\rm meas}\left\{\tau\in[0,T]\;\left|\;                  
\sup_{s\in K}|\zeta(s+i\tau)-\zeta(s+id\tau)|<\varepsilon\right.\right\}>0
\end{align}
holds if $d$ is algebraic irrational.
Nakamura also proved in the same paper that \eqref{9-3} is valid for almost all
$d\in\mathbb{R}$.   (Note that \eqref{9-3} for $d=0$ is \eqref{9-1}, hence the
Riemann hypothesis.)

Nakamura's paper sparked off the interest in this direction of research; 
Pa{\'n}kowski \cite{Pan09} proved that \eqref{9-3} holds for all (algebraic and
transcendental) irrational $d$, using the six exponentials theorem in
transcendental number theory.   On the other hand, Garunk{\v s}tis \cite{Gar11} and
Nakamura \cite{Nak10}, independently, claimed that \eqref{9-3} holds for all
non-zero rational.   However their arguments included a gap, which was partially
filled by Nakamura and Pa{\'n}kowski \cite{NakPan12}.   The present situation is:

\begin{thm}\label{thm9-2}
{\rm (Garunk{\v s}tis, Nakamura, Pa{\'n}kowski)}
The inequality \eqref{9-3} holds if $d$ is irrational, or $d=a/b$ is non-zero
rational with $(a,b)=1$, $|a-b|\neq 1$.
\end{thm}

See also Mauclaire \cite{Mau13}, and Nakamura and Pa{\'n}kowski \cite{NakPan13}.
It is to be noted that the argument of Garunk{\v s}tis \cite{Gar11} and
Nakamura \cite{Nak10} is correct for $\log\zeta(s)$, so
\begin{align}\label{9-4}                                                               
\liminf_{T\to\infty}\frac{1}{T}{\rm meas}\left\{\tau\in[0,T]\;\left|\;                  
\sup_{s\in K}|\log\zeta(s+i\tau)-\log\zeta(s+id\tau)|<\varepsilon\right.\right\}>0      
\end{align}
can be shown for any non-zero $d\in\mathbb{R}$.   If \eqref{9-4} would be valid for
$d=0$, it would imply the Riemann hypothesis.

The strong recurrence property can be shown for more general zeta and $L$-functions.
Some of the aforementioned papers actually consider not only the Riemann
zeta-function, but also Dirichlet $L$-functions.   A generalization to a subclass
of the Selberg class was discussed by Nakamura \cite{Nak11b}.  The case of
Hurwitz zeta-functions has been studied by Garunk{\v s}tis and Karikovas
\cite{GarKarpre} and Karikovas and Pa{\'n}kowski \cite{KarPanpre}.

\section {The weighted universality}\label{sec-10}

Laurin{\v c}ikas \cite{Lau95} considered a weighted version of the universality for
$\zeta(s)$.   Let $w(t)$ be a positive-valued function of bounded variation
defined on $[T_0,\infty)$ (where $T_0>0$), satisfying that the variation on
$[a,b]$ does not exceed $cw(a)$ with a certain $c>0$ for any subinterval
$[a,b]\subset[T_0,\infty)$.    Define
$$
U(T,w)=\int_{T_0}^T w(t)dt,
$$
and assume that $U(T,w)\to\infty$ as $T\to\infty$.

We further assume the following property of $w(t)$, connected with ergodic theory.
Let $X(\tau,\omega)$ be any ergodic process defined on a certain probability 
space $\Omega$, $\tau\in\mathbb{R}$, $\omega\in\Omega$, 
$E(|X(\tau,\omega)|)<\infty$, and sample paths are Riemann
integrable almost surely on any finite interval.   Assume that
\begin{align}\label{10-1}
\frac{1}{U(T,w)}\int_{T_0}^T w(\tau)X(t+\tau,\omega)d\tau=E(X(0,\omega))
+o((1+|t|)^{\alpha})
\end{align}
almost surely for any $t\in\mathbb{R}$, with an $\alpha>0$, as $T\to\infty$.
Denote by $I(A)$ the indicator function of the set $A$.

\begin{thm}\label{thm10-1}
{\rm (Laurin{\v c}ikas \cite{Lau95})}
Suppose that $w(t)$ satisfies all the above conditions.   Let $K$ be a compact
subset of $D(1/2,1)$ with connected complement, $f\in H_0^c(K)$.
Then
\begin{align}\label{10-2}
\lefteqn{\liminf_{T\to\infty}\frac{1}{U(T,w)}\int_{T_0}^T w(\tau)}\\
&\times I\left(\left\{\tau\in[T_0,T]\;\left|\;\sup_{s\in K}|\zeta(s+i\tau)-f(s)|        
<\varepsilon\right.\right\}\right)
d\tau >0 \notag
\end{align}
holds for any $\varepsilon>0$.
\end{thm}

In the course of Bagchi's proof of the universality theorem, there is a point
where the Birkhoff-Khinchin theorem
\begin{align}\label{10-3}
\lim_{T\to\infty}\frac{1}{T}\int_0^T X(\tau,\omega)d\tau=E(X(0,\omega))
\end{align}
in ergodic theory is used.   This is the motivation of
Laurin{\v c}ikas \cite{Lau95}; clearly \eqref{10-1} is a generalization of
\eqref{10-3}.

Laurin{\v c}ikas \cite{Lau98} generalized Theorem \ref{thm10-1} to the case of
Matsumoto zeta-functions.   Weighted universality theorems for $L$-functions
$L(s,E)$ of elliptic curves over $\mathbb{Q}$ were reported by
Garbaliauskien{\.e} \cite{Garb04} \cite{Garb05}.

\section{The discrete universality}\label{sec-11}

In the previous sections, we discussed the behaviour of zeta or $L$-functions
when the imaginary part $\tau$ of the variable is moving continuously.
However, we can also obtain a kind of universality theorems when $\tau$ only
moves discretely.   We already mentioned in Section \ref{sec-1} that Voronin's
multi-dimensional denseness theorem is valid in this sense (see Remark \ref{rem1-1}).

The first {\it discrete universality} theorem is due to Reich \cite{Rei80} on
Dedekind zeta-functions.   Let $F$ be a number field and $\zeta_F(s)$ the
associated Dedekind zeta-function.    

\begin{thm}\label{thm11-1}
{\rm (Reich \cite{Rei80})}
Let $K$ be a compact subset of 
the region $D(1-\max\{2,d_F\}^{-1},1)$ with connected complement, 
and $f\in H_0^c(K)$.
Then, for any real $h\neq 0$ and any $\varepsilon>0$,
\begin{align}\label{11-1}
\liminf_{N\to\infty}\frac{1}{N}\#\left\{n\leq N\;\left|\;\sup_{s\in K}
|\zeta_F(s+ihn)-f(s)|<\varepsilon\right.\right\}>0.
\end{align}
\end{thm}

The joint discrete universality theorem for Dirichlet $L$-functions was given in
Bagchi \cite{Bag81}.   He also obtained the discrete universality for 
$\zeta(s,\alpha)$ when $\alpha$ is rational, for which 
Sander and J. Steuding \cite{SanSte06} gave
a different approach.

Ka{\v c}inskait{\.e} \cite{Kac02} (see also \cite{Kac12})
proved a discrete universality theorem for
Matsumoto zeta-functions under the condition that $\exp(2\pi k/h)$ is
irrational for any non-zero integer $k$.   Ignatavi{\v c}i{\=u}t{\.e} \cite{Ign02} 
reported
that certain discrete universality  and certain joint discrete universality hold for 
Lerch zeta-functions, provided $\exp(2\pi/h)$ is rational.

As can be seen in the above examples, an interesting point on discrete universality
is that the arithmetic nature of the parameter $h$ plays a role.

The discrete universality for $L$-functions $L(s,E)$ of elliptic curves 
was first studied in \cite{GarbLau04} under the assumption that $\exp(2\pi k/h)$ is
irrational for any non-zero integer $k$.   However this condition was then
removed:

\begin{thm}\label{thm11-2}
{\rm (Garbaliauskien{\.e}, Genys and Laurin{\v c}ikas \cite{GarbGenLau08})}
The discrete universality holds for $L(s,E)$ for any real $h\neq 0$
in the region $D(1,3/2)$.
\end{thm}

The same type of result can be shown, more generally, for $L$-functions attached 
to new forms.   When $\exp(2\pi k/h)$ is
irrational for any non-zero integer $k$, this was done in 
Laurin{\v c}ikas, Matsumoto and J. Steuding \cite{LauMatSte05}.

The discrete universality for periodic zeta-functions was studied in
Ka{\v c}inskait{\.e}, Javtokas and {\v S}iau{\v c}i{\=u}nas \cite{KacJavSia08}
and Laurin{\v c}ikas, Macaitien{\.e} and {\v S}iau{\v c}i{\=u}nas \cite{LauMacSia09}, 
while the case of periodic Hurwitz zeta-functions was discussed by
Laurin{\v c}ikas and Macaitien{\.e} \cite{LauMac09b}.   The result in
\cite{LauMac09b} especially includes the discrete universality of the Hurwitz
zeta-function $\zeta(s,\alpha)$ when $\alpha$ is transcendental.
Laurin{\v c}ikas \cite{LauJNT} further proved that if the set
$$                                                                                      
\{\log(m+\alpha)\;|\;m\in\mathbb{N}_0\}\cup\{2\pi/h\}   
$$
is linearly independent over ${\mathbb Q}$, then the discrete universality holds
for $\zeta(s,\alpha)$.   See also \cite{BuiLauMacRas14}.
A joint version is studied in Laurin{\v c}ikas \cite{LauMS}.

Macaitien{\.e} \cite{Mac06} obtained a discrete universality theorem for general
Dirichlet series.   

\begin{thm}\label{thm11-3}
{\rm (Macaitien{\.e} \cite{Mac06})}
Let $f(s)$ be general Dirichlet series as in Theorem \ref{thm5-3}, and further
suppose that $\lambda_n$ are algebraic numbers and $\exp(2\pi/h)\in\mathbb{Q}$.
Then the discrete universality holds for $f(s)$ in the region 
$D(\sigma_1,\sigma_a)$.
\end{thm}

The discrete analogue of mixed universality can also be considered.
This direction was first studied by Ka{\v c}inskait{\.e} \cite{Kac11}.
Consider the case $\exp(2\pi/h)\in\mathbb{Q}$.   Write $\exp(2\pi/h)=a/b$
with $a,b\in{\mathbb Z}$ and $(a,b)=1$.   Denote by $P_h$ the set of all prime
numbers appearing as a prime factor of $a$ or $b$.   Define the modified
Dirichlet $L$-function $L_h(s,\chi)$ by removing all Euler factors 
corresponding to primes in $P_h$, that is
$$
L_h(s,\chi)=\prod_{p\notin P_h}\left(1-\frac{\chi(p)}{p^s}\right)^{-1}.
$$

\begin{thm}\label{thm11-4}
{\rm (Ka{\v c}inskait{\.e} \cite{Kac11})}
Let $K_1, K_2$ be compact subsets of $D(1/2,1)$ with connected complements,
and $f_1\in H_0^c(K_1)$, $f_2\in H^c(K_2)$.
If $\alpha$ is transcendental, $\mathfrak{B}$ is a periodic sequence and 
$\exp(2\pi/h)\in\mathbb{Q}$ as above, then
\begin{align}\label{11-2}
\liminf_{N\to\infty}\frac{1}{N}\#\left\{n\leq N\;\left|\;\sup_{s\in K_1}
|L_h(s+ihn,\chi)-f_1(s)|<\varepsilon,\right.\right.\\
\left.\sup_{s\in K_2}|\zeta(s+ihn,\alpha,\mathfrak{B})-f_2(s)|<\varepsilon\right\}>0
\notag
\end{align}
for any $\varepsilon>0$
\footnote{The statement in \cite{Kac11} is given for $L(s,\chi)$, but her
argument is valid not for $L(s,\chi)$, but for $L_h(s,\chi)$.}
.
\end{thm}

Buivytas and Laurin{\v c}ikas \cite{BuiLau} proved that if the set
$$
\{\log p\;|\;p:{\rm prime}\}\cup\{\log(m+\alpha)\;|\;m\in\mathbb{N}_0\}
\cup\{2\pi/h\}
$$
is linearly independent over ${\mathbb Q}$, then the discrete mixed universality
holds for $\zeta(s)$ and $\zeta(s,\alpha)$, that is,
\begin{align}\label{11-3}                                                               
\liminf_{N\to\infty}\frac{1}{N}\#\left\{n\leq N\;\left|\;\sup_{s\in K_1}                
|\zeta(s+ihn)-f_1(s)|<\varepsilon,\right.\right.\\                                   
\left.\sup_{s\in K_2}|\zeta(s+ihn,\alpha)-f_2(s)|<\varepsilon\right\}>0.    
\notag                                                                                  
\end{align}

In \eqref{11-2} and \eqref{11-3}, the shifting parameter $h$ is common to the
both
of relevant zeta (or $L$)-functions.   Buivytas and Laurin{\v c}ikas \cite{BuiLau2}
studied the case when the parameter for $\zeta(s)$ is different from the parameter
for $\zeta(s,\alpha)$.

We will encounter a rather different type of discrete universality theorems in
Section \ref{sec-18}.

\section{The $\chi$-universality}\label{sec-12}

The main point of Voronin's universality theorem is the existence of $\tau$,
the imaginary part of the complex variable $s$, satisfying a certain
approximation condition.   This is the common feature of all universality
theorems mentioned in the previous sections.   However there is another type of
universality, which is concerned with the existence of a character satisfying
certain approximation conditions.

The first theorem in this direction is as follows.

\begin{thm}\label{thm12-1}
{\rm (Bagchi \cite{Bag81}, Gonek \cite{Gon79}, Eminyan \cite{Emi90})}
Let $K$ be a compact subset of $D(1/2,1)$ with connected complement, and
$f\in H_0^c(K)$.   Let $Q$ be an infinite set of positive integers.
Then for any $\varepsilon>0$, 
\begin{align}\label{12-1}
\liminf_{q\to\infty\atop q\in Q}\frac{1}{\varphi(q)}
\;\#\left\{\chi {\rm (mod}\;q{\rm )}\;\left|\;\sup_{s\in K}
|L(s,\chi)-f(s)|<\varepsilon\right.\right\}>0
\end{align}
holds, provided $Q$ is one of the following\footnote
{Bagchi \cite[Theorem 5.3.11]{Bag81} proved the case (i).   In the statement of
Gonek \cite[Theorem 5.1]{Gon79} there is no restriction on $Q$, but Mishou
\cite{Mis03b} pointed out that condition (ii) is necessary to verify Gonek's proof.
Eminyan \cite{Emi90} studied the special case $r=1$ of (ii).}:

{\rm (i)} $Q$ is the set of all prime numbers;

{\rm (ii)} $Q$ is the set of positive integers of the form
$q=p_1^{a_1}\cdots p_r^{a_r}$ {\rm(}$a_1,\ldots,a_r\in\mathbb{N}\cup\{0\}${\rm)},
where $\{p_1,\ldots,p_r\}$ is a fixed finite set of prime numbers.
\end{thm}

This type of results is called the {\it $\chi$-universality}, or 
{\it the universality in $\chi$-aspect}.   The universality for Hecke
$L$-functions of number fields in $\chi$-aspect was discussed by 
Mishou and Koyama \cite{MisKoy02}, and by Mishou \cite{Mis04} \cite{Mis05}.

Let $\chi_d$ be a real Dirichlet character with discriminant $d$.
Another interesting direction of research is the universality for $L(s,\chi_d)$ 
in $d$-aspect.   
This direction was studied in a series of papers by Mishou and Nagoshi
\cite{MisNag06} \cite{MisNag06b} \cite{MisNag06c} \cite{MisNag12}.
Let $\Lambda^+$ (resp. $\Lambda^-$) be the set of all positive (resp. negative)
discriminants, and $\Lambda^+(X)$ (resp. $\Lambda^-(X)$) be the set of 
discriminants $d$ satisfying $0<d\leq X$ (resp. $-X\leq d<0$).

\begin{thm}\label{thm12-2}
{\rm (Mishou and Nagoshi \cite{MisNag06})}
Let $\Omega$ be a simply connected domain in $D(1/2,1)$ which is symmetric with
respect to the real axis.
Let $f(s)$ be holomorphic and non-vanishing on $\Omega$, and positive-valued on
$\Omega\cap\mathbb{R}$.   Let $K$ be a compact subset of $\Omega$.
Then, for any $\varepsilon>0$, 
\begin{align}\label{12-2}
\liminf_{X\to\infty}\frac{1}{\# \Lambda^{\pm}(X)}\#\left\{d\in \Lambda^{\pm}(X)
\;\left|\;\sup_{s\in K}|L(s,\chi_d)-f(s)|<\varepsilon\right.\right\}>0
\end{align}
holds.
\end{thm}

This theorem especially
implies that for any $s\in D(1/2,1)\setminus\mathbb{R}$, the set
$\{L(s,\chi_d)\;|\;d\in\Lambda^{\pm}\}$ is dense in $\mathbb{C}$, and for any 
real number $\sigma$ with $1/2<\sigma<1$, the set 
$\{L(\sigma,\chi_d)\;|\;d\in\Lambda^{\pm}\}$ is dense in the set of
positive real numbers $\mathbb{R}_+$.

In the same paper Mishou and Nagoshi also studied the situation on the line
$\Re s=1$, and proved that the set $\{L(1,\chi_d)\;|\;d\in\Lambda^{\pm}\}$
is dense in $\mathbb{R}_+$.   Therefore we can deduce denseness results on class
numbers of quadratic fields.   Let $h(d)$ be the class number of 
$\mathbb{Q}(\sqrt{d})$, and when $d>0$, let $\varepsilon(d)$ be the fundamental
unit of $\mathbb{Q}(\sqrt{d})$.   Then the above result implies that both the sets
$$
\left\{\frac{h(d)\log\varepsilon(d)}{\sqrt{d}}\;\left|\;d\in\Lambda^+\right.
\right\}, \qquad
\left\{\frac{h(d)}{\sqrt{d}}\;\left|\;d\in\Lambda^-\right.\right\}
$$
are dense in $\mathbb{R}_+$.

In \cite{MisNag06b} \cite{MisNag12}, the same type of problem for prime discriminants
are studied.   As an application, in \cite{MisNag12} Mishou and Nagoshi gave a
quantitative result on a problem of Ayoub, Chowla and Walum on certain character
sums \cite{AyoChoWal67}.   
In \cite{MisNag06c} Mishou and Nagoshi gave some conditions equivalent to
the Riemann hypothesis from their viewpoint.

The universality in $d$-aspect for Hecke $L$-functions of class group characters
for imaginary quadratic fields are studied by Mishou \cite{Mis11}.
In \cite{Mis09}, Mishou considered cubic characters associated with the field
$\mathbb{Q}(\sqrt{-3})$, and proved a universality theorem for associated Hecke
$L$-functions in cubic character aspect.

\section{Other applications}\label{sec-13}

So far we mentioned applications of universality to the theory of distribution
of zeros (Section \ref{sec-5}), to the Riemann hypothesis (Section \ref{sec-9}),
and to algebraic number theory (Section \ref{sec-12}).
Those applications, however, do not exhaust the potentiality of universality
theory.   In this section we discuss other applications of universality.

In Section \ref{sec-1} we mentioned that Voronin, before proving his
universality theorem, obtained the multi-dimensional denseness theorem (Theorem
\ref{thm1-2}) of $\zeta(s)$ and its derivatives.   However, now, we can say that
Theorem \ref{thm1-2} is just an immediate consequence of
the universality theorem (see \cite[Theorem 6.6.2]{Lau96}).
Moreover, from Theorem \ref{thm1-2} it is easy to obtain the following
functional-independence property of $\zeta(s)$.

\begin{thm}\label{thm13-1}
Let $f_l:\mathbb{C}^m\to\mathbb{C}$ $(0\leq l\leq n)$ be continuous functions, and 
assume that the equality
\begin{align}\label{13-1}
\sum_{l=0}^n s^l f_l(\zeta(s),\zeta^{\prime}(s),\ldots,\zeta^{(m-1)}(s))=0
\end{align}
holds for all $s$.   Then $f_l\equiv 0$ for $0\leq l\leq n$.
\end{thm}

This type of result was first noticed by Voronin himself
(\cite{Vor75b} \cite{Vor85}); see also \cite[Theorem 6.6.3]{Lau96}.

When $f_l$'s are polynomials, then \eqref{13-1} is an algebraic differential
equation.   Therefore Theorem \ref{thm13-1} in this case implies that $\zeta(s)$ 
does not satisfy any non-trivial algebraic differential equation.   This property was
already noticed by Hilbert in his famous address \cite{Hil00} at 
the 2nd International
Congress of Mathematicians (Paris, 1900).    Theorem \ref{thm13-1} is a
generalization of this algebraic-independence property.

Similarly to the case of $\zeta(s)$, if a Dirichlet series $\varphi(s)$ is
universal, it is easy to prove the theorems analogous to Theorem \ref{thm1-2} and
Theorem \ref{thm13-1} for $\varphi(s)$.

An application of the universality to the problem on Dirichlet polynomials
was done by Andersson \cite{And99}.
He used the universality theorem to show that several
conjectures on lower bounds of certain integrals of Dirichlet polynomials, 
proposed by Ramachandra \cite{Ram92} and Balasubramanian and Ramachandra
\cite{BalRam95} are false.

The universality property was applied even in physics; see Gutzwiller
\cite{Gut83}, Bitar, Khuri and Ren \cite{BKR91}.

\section{The general notion of universality}\label{sec-14}

The main theme of the present article is the universality for zeta and 
$L$-functions.    However, the notion of universality was first introduced in
mathematics, under a very different motivation.

The first discovery of the universality phenomenon is usually attributed to
M. Fekete (1914/15, reported in \cite{Pal}), 
who proved that there exists a real power series 
$\sum_{n=1}^{\infty}a_n x^n$ such that, for any continuous
$f:[-1,1]\to\mathbb{R}$ with $f(0)=0$ we can choose positive integers 
$m_1,m_2,\ldots$ for which
$$
\sum_{n=1}^{m_k}a_n x^n\to f(x) \qquad(k\to\infty)
$$
holds uniformly on $[-1,1]$.   The proof is based on Weierstrass' approximation
theorem.

G. D. Birkhoff \cite{Bir29} proved that there exists an entire function $\psi(z)$ such
that, for any entire function $f(z)$, we can choose complex numbers $a_1,a_2,\ldots$ 
for which $\psi(z+a_k)\to f(z)$ (as $k\to\infty$) uniformly in any compact subset of
$\mathbb{C}$.

The terminology ``universality'' was first used by Marcinkiewicz \cite{Mar35}.
Various functions satisfying some property similar to those discovered by Fekete and
Birkhoff are known.   However, before the work of Voronin \cite{Vor75}, all of
those functions were constructed very artificially.    So far the class of zeta
and $L$-functions is the only ``natural'' class of functions for which the
universality property can be proved.   For the more detailed history of this
general notion of universality, see Grosse-Erdmann \cite{Gro99} and
Steuding \cite[Appendix]{Ste07}
\footnote{In the same Appendix, Steuding mentioned a $p$-adic version of Fekete's
theorem, which was originally proved in \cite{Ste02}.}
.

It is to be noted that the real origin of the whole theory is
Riemann's theorem that a conditionally convergent series can be
convergent (or divergent) to any value after some suitable rearrangement.
In fact, Fekete's result may be regarded as an analogue of Riemann's theorem
for continuous functions, while Pecherski{\u\i}'s theorem \cite{Pec73}
(mentioned in Section \ref{sec-1} as an essential tool in Voronin's proof)
gives an analogue of Riemann's theorem in Hilbert spaces.

A very general definition of universality was proposed by Grosse-Erdmann 
\cite{Gro87}\cite{Gro99}.

\begin{defin}\label{defin14-1}
Let $X$, $Y$ be topological spaces, $W$ be a non-empty closed subset of $Y$, 
and $T_{\tau}:X\to Y$ ($\tau\in I$) be a family
of mappings with the index set $I$.   We call $x\in X$ 
{\it universal with respect to $W$} if the closure of the set
$\{T_{\tau}(x)\;|\;\tau\in I\}$ contains $W$.
\end{defin}

Let $K$ be as in Theorem \ref{thm1-1}, $X=Y=H(K^{\circ})$, where $K^{\circ}$ is
the interior of $K$.
Then obviously $H_0^c(K)\subset H(K^{\circ})$.
Put $W=\overline{H_0^c(K)}$ (the topological closure of $H_0^c(K)$ in the space
$H(K^{\circ})$).    Define $T_{\tau}$ by
$T_{\tau}(f(z))=f(z+i\tau)$ for $f\in H(K^{\circ})$.
Then Theorem \ref{thm1-1} implies that any element of $H_0^c(K)$ can be
approximated by some suitable element of 
$\{T_{\tau}(\zeta)\;|\;\tau\in \mathbb{R}\}$.   Therefore Theorem \ref{thm1-1}
asserts that the Riemann zeta-function $\zeta(s)$ is universal with respect to
$\overline{H_0^c(K)}$ in the sense of Definition \ref{defin14-1}.

The notion of joint universality can also be formulated in this general setting.

\begin{defin}\label{defin14-2}
Let $X,Y_1,\ldots,Y_r$ be topological spaces, and
$T_{\tau}^{(j)}:X\to Y_j$ ($\tau\in I$, $1\leq j\leq r$) be families of mappings.
We call $x_1,\ldots,x_r\in X$ {\it jointly universal} if the set
$\{(T_{\tau}^{(1)}(x_1),\ldots,T_{\tau}^{(r)}(x_r)\;|\;\tau\in I\}$ 
is dense in $Y_1\times\cdots\times Y_r$.
\end{defin}

\begin{rem}\label{rem14-1}
The case $r=1$ of Definition \ref{defin14-2} is the case $W=Y$ of Definition 
\ref{defin14-1}.
\end{rem}

In Section \ref{sec-1} we mentioned that the Kronecker-Weyl approximation theorem
(see Remark \ref{rem1-0}) is used in the proof of Theorem \ref{thm1-1}.    
We can see that the
Kronecker-Weyl theorem itself implies a certain universality phenomenon.
Let $S^1=\{z\in\mathbb{C}\;|\;|z|=1\}$, and consider the situation when
$X=\mathbb{R}$, $Y_1=\cdots=Y_r=S^1$ in Definition \ref{defin14-2}.
Define $T_{\tau}^{(1)}=\cdots=T_{\tau}^{(r)}=T_{\tau}:\mathbb{R}\to S^1$ by
$T_{\tau}(x)=e^{2\pi i\tau x}$.   Then the Kronecker-Weyl theorem implies that,
if $\alpha_1,\ldots,\alpha_r\in\mathbb{R}$ are linearly independent over
$\mathbb{Q}$, then the orbit of
$$
(T_{\tau}(\alpha_1),\ldots,T_{\tau}(\alpha_r))
=(e^{2\pi i\tau\alpha_1},\ldots,e^{2\pi i\tau\alpha_r})
$$
is dense in $S^1\times\cdots\times S^1$.   Therefore 
$\alpha_1,\ldots,\alpha_r$ are jointly universal.

The above observation shows that both the Voronin theorem and the
Kronecker-Weyl theorem express certain universality properties.
Is it possible to combine these two universality theorems?
The answer is yes, and we will discuss this matter in the next section.

\section{The hybrid universality}\label{sec-15}

The first affirmative answer to the question raised at the end of Section 
\ref{sec-14} was given by Gonek \cite{Gon79}, and a slightly general result
was later obtained by Kaczorowski and Kulas \cite{KacKul07}.

\begin{thm}\label{thm15-1}
{\rm (Gonek \cite{Gon79}, Kaczorowski and Kulas \cite{KacKul07})}
Let $K$ be a compact subset of $D(1/2,1)$, $f_1,\ldots,f_r\in H_0^c(K)$,
$\chi_1,\ldots,\chi_r$ be pairwise non-equivalent Dirichlet characters,
$z>0$, and $(\theta_p)_{p\leq z}$ be a sequence of real numbers indexed by
prime numbers up to $z$.   Then, for any $\varepsilon>0$, 
\begin{align}\label{15-1}
\liminf_{T\to\infty}\frac{1}{T}{\rm meas}\left\{\tau\in[0,T]\;\left|\;                  
\max_{1\leq j\leq r}\sup_{s\in K}|L(s+i\tau,\chi_j)-f_j(s)|<\varepsilon,\right.
\right.\\
\left.\max_{p\leq z}\left|\left|\tau\frac{\log p}{2\pi}-\theta_p\right|\right|
<\varepsilon\right\}>0 \notag
\end{align}
holds.
\end{thm}

The combination of the universality of Voronin type and of Kronecker-Weyl type
is now called the {\it hybrid universality}.   The above theorem is therefore an
example of the hybrid joint universality.

Pa{\'n}kowski \cite{Pan10} proved that the second inequality in \eqref{15-1} can be
replaced by $\max_{1\leq k\leq m}||\tau\alpha_k-\theta_k||<\varepsilon$, where
$\alpha_1,\ldots,\alpha_m$ are real numbers which are linearly independent over
$\mathbb{Q}$, and $\theta_1,\ldots,\theta_m$ are arbitrary real numbers.
This is exactly the same inequality as in the Kronecker-Weyl theorem (see
Remark \ref{rem1-0}).

In the same paper Pa{\'n}kowski remarked that the same statement can be shown 
for more general $L$-functions which have Euler products.   The hybrid joint
universality for some zeta-functions without Euler products was discussed in
Pa{\'n}kowski \cite{Pan13}.

Hybrid universality theorems are quite useful in applications.
Gonek \cite{Gon79} used his hybrid universality theorem to show the joint universality
for Dedekind zeta-functions of Abelian number fields (mentioned in Section
\ref{sec-3}).   The key fact here is that those Dedekind zeta-functions can be
written as products of Dirichlet $L$-functions.   On the other hand, the aim of 
Kaczorowski and Kulas \cite{KacKul07} was to study the distribution of zeros of
linear combinations of the form $\sum P_j(s)L(s,\chi_j)$, where $P_j(s)$ are
Dirichlet polynomials.   Kaczorowski and Kulas applied Theorem \ref{thm15-1} to
show a theorem\footnote{This theorem was later sharpened by Ki and Lee
\cite{KiLee11} by using the method of mean motions \cite{JesTor45} 
\cite{BorJes48}.}, similar to Theorem \ref{thm5-2}, for such linear combinations.

Sander and J. Steuding \cite{SanSte06} also considered the universality for sums, or
products of Dirichlet $L$-functions.   In particular they proved the joint
universality for Hurwitz zeta-functions $\zeta(s,a/q)$ ($1\leq a\leq q$) under
a certain condition on target functions.
Hurwitz zeta-functions usually do not have Euler products, but when $a=q$
(and when $q$ is even and $a=q/2$) the corresponding Hurwitz zeta-function is
essentially the Riemann zeta-function and hence has the Euler product.
Therefore the result of Sander and J. Steuding is an example of mixed universality
(see Section \ref{sec-8}).

In the paper of Kaczorowski and Kulas \cite{KacKul07}, the coefficients $P_j(s)$
of linear combinations are Dirichlet polynomials.   Nakamura and Pa{\'n}kowski
\cite{NakPan12b} considered a more general situation when $P_j(s)$ are Dirichlet 
series.  Their general statement is as follows.

\begin{thm}\label{thm15-2}
{\rm (Nakamura and Pa{\'n}kowski \cite{NakPan12b})}
Let $P_1(s),\ldots,P_r(s)$ ($r\geq 2$) be general Dirichlet series, not identically
vanishing, absolutely convergent in $\Re s>1/2$.   Moreover assume that at least
two of those are non-vanishing in $D(1/2,1)$.   Let $L_1(s),\ldots,L_r(s)$
be hybridly jointly universal in the above sense.   Then
$L(s)=\sum_{j=1}^rP_j(s)L_j(s)$   
is strongly universal in $D(1/2,1)$.
\end{thm} 

As a corollary, by Theorem \ref{thm5-2} we find $N(T;\sigma_1,\sigma_2;L)\geq CT$
for the above $L(s)$, for any $\sigma_1,\sigma_2$ satisfying
$1/2<\sigma_1<\sigma_2<1$.

In the above theorem $L(s)$ is a linear form of $L_j$'s, but in
\cite{NakPan11} \cite{NakPanpre}, Nakamura and Pa{\'n}kowski obtained more general
statements; they considered polynomials of $L_j$'s whose coefficients are general 
Dirichlet series, and proved results similar to Theorem \ref{thm15-2}.
Note that when coefficients of polynomials are constants, such a result was
already given in Ka{\v c}inskait{\.e}, J. Steuding, {\v S}iau{\v c}i{\=u}nas and
{\v S}le{\v z}evi{\v c}ien{\.e} \cite{KacSteSiaSle04}.

Many important zeta and $L$-functions have such polynomial expressions.
Consequently, Nakamura and Pa{\'n}kowski succeeded in proving the inequalities
like \eqref{5-3} on the distribution of zeros of those zeta or $L$-functions,
such as zeta-functions attached to symmetric matrices (in the theory of
prehomogeneous vector spaces), Estermann zeta-functions,
Igusa zeta-functions associated with local Diophantine problems, 
spectral zeta-functions associated with Laplacians on Riemannian manifolds, 
Epstein zeta-functions (see \cite{NakPan13b}), and also various multiple
zeta-functions (of Euler-Zagier, of Barnes, of Shintani, of Witten and so on).

\section{Quantitative results}\label{sec-16}

It is an important question how to obtain quantitative information related with
universality.   For example, let
\begin{align}\label{16-1}
d(\zeta,f,K,\varepsilon)=\liminf_{T\to\infty}\frac{1}{T}{\rm meas}
\left\{\tau\in[0,T]\;\left|\;\sup_{s\in K}|\zeta(s+i\tau)-f(s)|<\varepsilon
\right.\right\}.
\end{align}

Theorem \ref{thm1-1} asserts that $d(\zeta,f,K,\varepsilon)$ is positive; but how 
to evaluate this value?   Or, how to find the smallest value of $\tau$
(which we denote by $\tau(\zeta,f,K,\varepsilon)$) satisfying
the inequality $\sup_{s\in K}|\zeta(s+i\tau)-f(s)|<\varepsilon$?
Voronin's proof gives no information on these questions, because in the course
of the proof Voronin used Pecherski{\u\i}'s rearrangement theorem, which is
ineffective.

The first attempt to get a quantitative version of the universality theorem is
due to Good \cite{Goo81}.   The fundamental idea of Good is to combine the
argument of Voronin with the method of Montgomery \cite{Mon77}, in which
Montgomery studied large values of $\log\zeta(s)$.   Instead of
Pecherski{\u\i}'s theorem, Good used convexity arguments (Hadamard's three circles
theorem and the Hahn-Banach theorem).   Also Koksma's quantitative version of
Weyl's criterion is invoked.   The statements of Good's main results are quite
complicated, but it includes a quantitative version of the discrete universality
theorem for $\zeta(s)$.

\begin{rem}\label{rem16-1}
Actually Good stated the discrete universality for $\log\zeta(s)$.   From the
universality for $\log\zeta(s)$, the universality for $\zeta(s)$ itself can be
immediately deduced by exponentiation.   A proof of the universality for 
$\log\zeta(s)$ by Voronin's original method is presented in the book of
Karatsuba and Voronin \cite[Chapter VII]{KarVor92}.
\end{rem}

Good's idea was further pursued by Garunk{\v s}tis \cite{Gar03}, who obtained
a more explicit quantitative result when $K$ is small.   
His main theorem is still rather complicated,
but as a typical special case, he showed the following inequalities.   Let 
$K=K(r)$ be as in Section \ref{sec-1}.

\begin{thm}\label{thm16-1}
{\rm (Garunk{\v s}tis \cite{Gar03})} 
Let $0<\varepsilon\leq 1/2$, $f\in H(K(0.05))$, and assume  
$\max_{s\in K(0.06)}|f(s)|\leq 1$.   Then we have
\begin{align*}
\begin{array}{ll}
d(\log\zeta,f,K(0.0001),\varepsilon)\geq \exp(-1/\varepsilon^{13}),\\
\tau(\log\zeta,f,K(0.0001),\varepsilon)\leq \exp\exp(10/\varepsilon^{13}).
\end{array}
\end{align*} 
\end{thm}

Besides the above work of Good and Garunk{\v s}tis,
there are various different approaches toward quantitative results.
Laurin{\v c}ikas \cite{Lau00} pointed out that quantitative information on the
speed of convergence of a certain functional limit theorem would give a
quantitative result on the universality for Lerch zeta-functions.
J. Steuding \cite{Ste03c} \cite{Ste05} considered the quantity defined by replacing
liminf on \eqref{1-3} by limsup, and discussed its upper bounds.

The author pointed out in \cite{Mat06} that from Theorem \ref{thm1-2},
by comparing the Taylor expansions of $\zeta(s+i\tau)$ and $f(s)$, it is possible
to deduce a certain weaker version of universal approximation.   On the other hand,
a quantitative version of Theorem \ref{thm1-2} was shown by Voronin himself
\cite{Vor88}.
Combining these two ideas, a quantitative version of weak universal approximation
theorem was obtained in Garunk{\v s}tis, Laurin{\v c}ikas, Matsumoto,
J. \& R. Steuding \cite{GarLauMatSteSte10}.

A nice survey on the effectivization problem is given in Laurin{\v c}ikas
\cite{Lau13b}.

\section{The universality for derived functions}\label{sec-17}

When some function $\varphi(s)$ satisfies the universality property, a natural question 
is to ask whether functions derived from $\varphi(s)$ by some standard operations,
such as $\varphi^{\prime}(s)$, $\varphi(s)^2$, $\exp(\varphi(s))$ etc, also
satisfy the universality property, or not.

We already mentioned the universality of $\log\zeta(s)$ (see Remark \ref{rem16-1}).
Concerning the derivatives, Bagchi \cite{Bag82} proved that $m$th derivatives of 
Dirichlet $L$-functions $L^{(m)}(s,\chi)$ ($m\in\mathbb{N}$) are strongly
universal.
Laurin{\v c}ikas \cite{Lau85} studied the universality for $(\zeta'/\zeta)(s)$,
and then considered the same problem for $L(s,g)$ (for a cusp form $g$) in
\cite{Lau05b} \cite{Lau05c}.   The universality for derivatives of $L$-functions
of elliptic curves was studied by Garbaliauskien{\.e} and  Laurin{\v c}ikas
\cite{GarbLau07}, and its discrete analogue was discussed by Belovas, 
Garbaliauskien{\.e} and Ivanauskait{\.e} \cite{BelGarbIva08}.

After these early attempts, Laurin{\v c}ikas \cite{Lau10c} 
(see also \cite{Lau12c}) formulated a more
general framework of {\it composite universality}.   Let $F$ be an operator
$F:H(D(1/2,1))\to H(D(1/2,1))$.   Laurin{\v c}ikas \cite{Lau10c} considered
when $F(\zeta(s))$ has the universality property.

A simple affirmative case is the Lipschitz class ${\rm Lip}(\alpha)$.   We call $F$
belongs to ${\rm Lip}(\alpha)$ when

1) for any
polynomial $q=q(s)$ and any compact $K\subset D(1/2,1)$, there exists
$q_0\in F^{-1}\{q\}$ such that $q_0(s)\neq 0$ on $K$, and

2) for any compact $K\subset D(1/2,1)$ with connected complement, there exist
$c>0$ and a compact $K_1\subset D(1/2,1)$ with connected complement, for which
$$
\sup_{s\in K}|F(g_1(s))-F(g_2(s))|\leq c\sup_{s\in K_1}|g_1(s)-g_2(s)|^{\alpha}
$$
holds for all $g_1,g_2\in H(D(1/2,1))$.

Then it is pointed out in \cite{Lau10c} that if $F\in {\rm Lip}(\alpha)$, then
$F(\zeta(s))$ has the universality property.   This claim especially includes
the proof of the universality for $\zeta'(s)$.

To prove other theorems in \cite{Lau10c}, Laurin{\v c}ikas applied the method of
functional limit theorems (cf. Mauclaire \cite{Mau13}).   Those results imply, 
for example, 
$\zeta'(s)+\zeta''(s)$, $\zeta(s)^m$ ($m$-th power), $\exp(\zeta(s))$,
and $\sin(\zeta(s))$ are universal.

Later, Laurin{\v c}ikas and his colleagues generalized the result in \cite{Lau10c}
to various other situations.

$\bullet$ Hurwitz zeta-functions (Laurin{\v c}ikas \cite{Lau13} for the single case, 
Laurin{\v c}ikas \cite{Lau12} for the joint case, and Laurin{\v c}ikas and
Ra{\v s}yt{\.e} \cite{LauRas12} for the discrete case),

$\bullet$ Periodic and periodic Hurwitz zeta-functions 
(Laurin{\v c}ikas \cite{Lau12b}, Korsakien{\.e}, Pocevi{\v c}ien{\.e} and
{\v S}iau{\v c}i{\=u}nas \cite{KorPocSia13}),

$\bullet$ Automorphic $L$-functions (Laurin{\v c}ikas, Matsumoto and J. Steuding
\cite{LauMatSte13}).

A hybrid version of the joint composite universality for Dirichlet $L$-functions
was given by Laurin{\v c}ikas, Matsumoto and J. Steuding \cite{LauMatSte13b}.

An alternative approach to composite universality was done by Meyrath \cite{Mey11}.

Yet another approach is due to Christ, J. Steuding and Vlachou \cite{ChrSteVla13}.
Let $\Omega_0\times\cdots\times\Omega_n$ be an open subset of $\mathbb{C}^{n+1}$,
and let $F:\Omega_0\times\cdots\times\Omega_n\to\mathbb{C}$ be continuous.
In \cite{ChrSteVla13}, the universality of 
$F(\zeta(s),\zeta'(s),\ldots,\zeta^{(n)}(s))$ was discussed.
First, applying the idea in \cite{GarLauMatSteSte10}, they proved a weaker form
of universal approximation.
Then, when $F$ is non-constant and analytic, they obtained a kind of universality
theorem on a certain small circle.   Their proof relies on the implicit function
theorem and the Picard-Lindel{\"o}f theorem on certain differential equations.

\section{Ergodicity and the universality}\label{sec-18}

It is quite natural to understand universality from the ergodic
viewpoint.   In fact, the universality theorem for a certain Dirichlet
series $\varphi(s)$ implies that the orbit 
$\{\varphi(s+i\tau)\;|\;\tau\in\mathbb{R}\}$ 
is dense in a certain function space, and this orbit comes back to an arbitrarily
small neighbourhood of any target function infinitely often.   This is really an
ergodic phenomenon.

Therefore, we can expect that there are some explicit connections between 
universality theory and ergodic theory.   We mentioned already in Section 
\ref{sec-10} that the Birkhoff-Khinchin theorem in ergodic theory is used
in Bagchi's proof of the universality.

Recently J. Steuding \cite{Ste13} formulated a kind of universality theorem,
whose statement itself is written in terms of ergodic theory.

Let $(X,\mathcal{B},P)$ be a probability space, and let $T:X\to X$ a
measure-preserving transformation.    We call $T$ ergodic with respect to $P$ if
$A\in\mathcal{B}$ satisfies $T^{-1}(A)=A$, then either $P(A)=0$ or $P(A)=1$
holds.   In this case we call $(X,\mathcal{B},P,T)$ an ergodic dynamical system.
Here we consider the case $X=\mathbb{R}$ and $\mathcal{B}$ is the standard Borel
$\sigma$-algebra.

Let $D\subset\mathbb{C}$ be a domain, $K_1,\ldots,K_r$ be compact subsets of
$D$ with connected complements, and $f_j\in H_0^c(K_j)$ ($1\leq j\leq r$).
We call $\varphi_1,\ldots,\varphi_r\in H(D)$ is {\it jointly ergodic universal}
if for any $K_j$, $f_j$, $T$, $\varepsilon>0$, and for almost all $x\in\mathbb{R}$,
there exists an $n\in\mathbb{N}$ for which
\begin{align}\label{18-1}
\max_{1\leq j\leq r}\sup_{s\in K_j}|\varphi_j(s+iT^n x)-f_j(s)|<\varepsilon
\end{align}
holds.    Here, $T^n x$ means the $T$-times iteration of $T$.   If the above
statament holds for $f_j\in H^c(K_j)$ ($1\leq j\leq r$), then we call
$\varphi_1,\ldots,\varphi_r\in H(D)$ {\it jointly strongly ergodic universal}.

\begin{thm}\label{thm18-1}
{\rm (J. Steuding \cite{Ste13})}
Let $D$, $K_j$, $T$ be as above.   Let $\varphi_1,\ldots,\varphi_r$ be a
family of $L$-functions.
Then, there exists a real number $\tau$ such that
\begin{align}\label{18-2}
\max_{1\leq j\leq r}\sup_{s\in K_j}|\varphi_j(s+i\tau)-f_j(s)|<\varepsilon
\end{align}
for any $f_j\in H_0^c(K_j)$ $({\rm resp.} \;H^c(K_j))$ and any 
$\varepsilon>0$, if and
only if $\varphi_1,\ldots,\varphi_r$ is jointly {\rm(}resp. jointly strongly{\rm)}
ergodic universal.   And in this case, we have
\begin{align}\label{18-3}
\liminf_{N\to\infty}\frac{1}{N}\#\left\{n\in\mathbb{N}\;\left|\;
\max_{1\leq j\leq r}\sup_{s\in K_j}|\varphi_j(s+iT^n x)-f_j(s)|<\varepsilon 
\right.\right\}>0.
\end{align}
\end{thm}

Srichan, R. \& J. Steuding \cite{SriSteSte13} discovered the universality 
produced by a random walk.   They considered a lattice $\Lambda$ 
on $\mathbb{C}$ and a random walk $(s_n)_{n=0}^{\infty}$ 
on this lattice, and proved the following result.   Let $K$ be a compact subset 
of $D(1/2,1)$ with connected complement (with a condition given in terms of
$\Lambda$), and $f\in H_0^c(K)$.   Then for any $\varepsilon>0$, 
\begin{align}\label{18-4}                                                               
\liminf_{N\to\infty}\frac{1}{N}\#\left\{n\in\mathbb{N}\;\left|\;                        
\sup_{s\in K}|\zeta(s+s_n)-f(s)|<\varepsilon             
\right.\right\}>0                                                                      
\end{align}
holds almost surely.   They also mentioned a result similar to the above for
two-dimensional Brownian motions.

The results in this section suggest that universality is a kind of ergodic
phenomenon and is to be understood from the viewpoint of dynamical systems.
Universality theorems imply that the properties of zeta-functions in the critical
strip are quite inaccessible, which is probably the underlying reason of the
extreme difficulty of the Riemann hypothesis.    Moreover in Section \ref{sec-9}
we mentioned that the Riemann hypothesis itself can be reformulated in terms of
dynamical systems.   Therefore the Riemann hypothesis is perhaps to be
understood as a phenomenon with dynamical-system flavour\footnote{This argument
might remind us the work of Deninger \cite{Den98} \cite{Den00} which is in a
different context but also with dynamical-system flavour.}.
In order to pursue this viewpoint, it is indispensable to study universality
more deeply and extensively.

%
%
\bigskip

\bigskip
\noindent Graduate School of Mathematics, \\
Nagoya University,\\
Chikusa-ku, Nagoya 464-8602,\\
Japan

\begin{thebibliography}{99}
\bibitem{And99}J. Andersson,
Disproof of some conjectures of K. Ramachandra,
Hardy-Ramanujan J. {\bf 22} (1999), 2-7.
\bibitem{AyoChoWal67}R. Ayoub, S. Chowla and H. Walum,
On sums involving quadratic characters,
J. London Math. Soc. {\bf 42} (1967), 152-154.
\bibitem{Bag81}B. Bagchi,
The statistical behaviour and universality properties of the Riemann 
zeta-function and other allied Dirichlet series, 
Thesis, Calcutta, Indian Statistical Institute, 1981.
\bibitem{Bag82}B. Bagchi,
A joint universality theorem for Dirichlet $L$-functions,
Math. Z. {\bf 181} (1982), 319-334.
\bibitem{Bag87}B. Bagchi,
Recurrence in topological dynamics and the Riemann hypothesis,
Acta Math. Hung. {\bf 50} (1987), 227-240.
\bibitem{BalRam95}R. Balasubramanian and K. Ramachandra,
On Riemann zeta-function and allied questions II,
Hardy-Ramanujan J. {\bf 18} (1995), 10-22.
\bibitem{Bau03}H. Bauer,
The value distribution of Artin $L$-series and zeros of zeta-functions,
J. Number Theory {\bf 98} (2003), 254-279.
\bibitem{BelGarbIva08}I. Belovas, V. Garbaliauskien{\.e} and R. Ivanauskait{\.e},
The discrete universality of the derivatives of $L$-functions of elliptic curves,
{\v S}iauliai Math. Semin. {\bf 3(11)} (2008), 53-59.
\bibitem{Bir29}G. D. Birkhoff,
D{\'e}monstration d'un th{\'e}or{`e}me {\'e}l{\'e}mentaire sur les fonctions
enti{\`e}res,
C. R. Acad. Sci. Paris {\bf 189} (1929), 473-475.
\bibitem{BKR91}K. M. Bitar, N. N. Khuri and H. C. Ren,
Path integrals and Voronin's theorem on the universality of the Riemann zeta
function,
Ann. Phys. {\bf 211} (1991), 172-196.
\bibitem{Boh15}H. Bohr,
Zur Theorie der Riemann'schen Zetafunktion im kritischen Streifen,
Acta Math. {\bf 40} (1915), 67-100.
\bibitem{Boh22}H. Bohr,
{\"U}ber eine quasi-periodische Eigenschaft Dirichletscher Reihen mit Anwendung
auf die Dirichletschen $L$-Funktionen,
Math. Ann. {\bf 85} (1922), 115-122.
\bibitem{BohCou14}H. Bohr and R. Courant,
Neue Anwendungen der Theorie der Diophantischen Approximationen auf die
Riemannschen Zetafunktion,
J. reine Angew. Math. {\bf 144} (1914), 249-274.
\bibitem{BohJes3032}H. Bohr and B. Jessen,
{\"U}ber die Werteverteilung der Riemannschen Zetafunktion, Erste Mitteilung,
Acta Math. {\bf 54} (1930), 1-35; Zweite Mitteilung, ibid. {\bf 58} (1932), 1-55.
\bibitem{BorJes48}V. Borchsenius and B. Jessen,
Mean motions and values of the Riemann zeta function,
Acta Math. {\bf 80} (1948), 97-166.
\bibitem{BuiLau}E. Buivydas and A. Laurin{\v c}ikas,
A discrete version of the Mishou theorem,
preprint.
\bibitem{BuiLau2}E. Buivydas and A. Laurin{\v c}ikas,
A discrete version of the joint universality theorem for the Riemann and Hurwitz 
zeta-functions,
preprint.
\bibitem{BuiLauMacRas14}E. Buivydas, A. Laurin{\v c}ikas, R. Macaitien{\.e} and
J. Ra{\v s}yt{\.e},
Discrete universality theorems for the Hurwitz zeta-function,
J. Approx. Theory {\bf 183} (2014), 1-13.
\bibitem{Cai02}Y. Cai,
Prime geodesic theorem,
J. Th{\'e}or. Nombr. Bordeaux {\bf 14} (2002), 59-72.
\bibitem{ChrSteVla13}T. Christ, J. Steuding and V. Vlachou,
Differential universality,
Math. Nachr. {\bf 286} (2013), 160-170.
\bibitem{Den98}C. Deninger,
Some analogies between number theory and dynamical systems on foliated spaces,
in ``Proc. Intern. Congr. Math. Berlin 1998'', Vol. I, Documenta Math. J. DMV
Extra Vol., 1998, pp.163-186.
\bibitem{Den00}C. Deninger,
On dynamical systems and their possible significance for arithmetic geometry,
in ``Regulators in Analysis, Geometry and Number Theory'',
A. Reznikov and N. Schappacher (eds.), Progr. in Math. {\bf 171}, Birkh{\"a}user,
2000, pp.29-87. 
\bibitem{DruGarKac13}P. Drungilas, R. Garunk{\v s}tis and A. Ka{\v c}{\.e}nas,
Universality of the Selberg zeta-function for the modular group,
Forum Math. {\bf 25} (2013), 533-564.
\bibitem{Dub12}A. Dubickas,
On the linear independence of the set of Dirichlet exponents,
Kodai Math. J. {\bf 35} (2012), 642-651.
\bibitem{Emi90}K. M. Eminyan,
$\chi$-universality of the Dirichlet $L$-function,
Mat. Zametki {\bf 47} (1990), 132-137 (in Russian);
Math. Notes {\bf 47} (1990), 618-622.
\bibitem{Garb04}V. Garbaliauskien{\.e},
A weighted universality theorem for zeta-functions of elliptic curves,
Liet. Mat. Rink. {\bf 44}, Spec. Issue (2004), 43-47.
\bibitem{Garb05}V. Garbaliauskien{\.e},
A weighted discrete universality theorem for $L$-functions of elliptic curves,
Liet. Mat. Rink. {\bf 45}, Spec. Issue (2005), 25-29.
\bibitem{GarbGenLau08}V. Garbaliauskien{\.e}, J. Genys and A. Laurin{\v c}ikas,
Discrete universality of the $L$-functions of elliptic curves,
Sibirski{\u\i} Mat. Zh. {\bf 49} (2008), 768-785 (in Russian);
Siberian Math. J. {\bf 49} (2008), 612-627. 
\bibitem{GarbKacLau04}V. Garbaliauskien{\.e}, R. Ka{\v c}inskait{\.e} and
A. Laurin{\v c}ikas,
The joint universality for $L$-functions of elliptic curves,
Nonlinear Anal. Modell. Control {\bf 9} (2004), 331-348.
\bibitem{GarbLau04}V. Garbaliauskien{\.e} and A. Laurin{\v c}ikas,
Discrete value-distribution of $L$-functions of elliptic curves,
Publ. Inst. Math. (Beograd) {\bf 76(90)} (2004), 65-71.
\bibitem{GarbLau05}V. Garbaliauskien{\.e} and A. Laurin{\v c}ikas,
Some analytic properties for $L$-functions of elliptic curves,
Proc. Inst. Math. NAN Belarus {\bf 13} (2005), 75-82.
\bibitem{GarbLau07}V. Garbaliauskien{\.e} and A. Laurin{\v c}ikas,
The universality of the derivatives of $L$-functions of elliptic curves,
in ``Analytic and Probabilistic Methods in Number Theory'' (Proc. 4th Palanga
Conf.), A. Laurin{\v c}ikas and E. Manstavi{\v c}ius (eds.), TEV, 2007,
pp.24-29.
\bibitem{Gar03}R. Garunk{\v s}tis,
The effective universality theorem for the Riemann zeta function,
in ``Proc. Session in Analytic Number Theory and Diophantine Equations'',
D. R. Heath-Brown and B. Z. Moroz (eds.), Bonner Math. Schriften {\bf 360},
Bonn, 2003, n.16, 21pp.
\bibitem{Gar11}R. Garunk{\v s}tis,
Self-approximation of Dirichlet $L$-functions,
J. Number Theory {\bf 131} (2011), 1286-1295.
\bibitem{GarKarpre}R. Garunk{\v s}tis and E. Karikovas,
Self-approximation of Hurwitz zeta-functions,
Funct. Approx. Comment. Math., to appear.
\bibitem{GarLauSleSte02}R. Garunk{\v s}tis, A. Laurin{\v c}ikas,
R. {\v S}le{\v z}evi{\v c}ien{\.e} and J. Steuding,
On the universality of Estermann zeta-functions, 
Analysis {\bf 22} (2002), 285-296.
\bibitem{GarLauMatSteSte10}R. Garunk{\v s}tis, A. Laurin{\v c}ikas, K. Matsumoto,
J. Steuding and R. Steuding,
Effective uniform approximation by the Riemann zeta-function,
Publ. Math. (Barcelona) {\bf 54} (2010), 209-219.
\bibitem{GenLau04}J. Genys and A. Laurin{\v c}ikas,
Value distribution of general Dirichlet series V,
Liet. Mat. Rink. {\bf 44} (2004), 181-195 (in Russian);
Lith. Math. J. {\bf 44} (2004), 145-156.
\bibitem{GenMacRacSia10}J. Genys, R. Macaitien{\.e}, S. Ra{\v c}kauskien{\.e}
and D. {\v S}iau{\v c}i{\=u}nas,
A mixed joint universality theorem for zeta-functions,
Math. Modell. Anal. {\bf 15} (2010), 431-446.
\bibitem{Gon79}S. M. Gonek,
Analytic properties of zeta and $L$-functions,
Thesis, University of Michigan, 1979.
\bibitem{Goo81}A. Good,
On the distribution of the values of Riemann's zeta-function,
Acta Arith. {\bf 38} (1981), 347-388.
\bibitem{Gro87}K.-G. Grosse-Erdmann,
Holomorphe Monster und universelle Funktionen,
Mitt. Math. Sem. Giessen {\bf 176} (1987), 1-81.
\bibitem{Gro99}K.-G. Grosse-Erdmann,
Universal families and hypercyclic operators,
Bull. Amer. Math. Soc. {\bf 36} (1999), 345-381.
\bibitem{Gut83}M. C. Gutzwiller,
Stochastic behavior in quantum scattering,
Physica {\bf 7D} (1983), 341-355.
\bibitem{Hil00}D. Hilbert,
Mathematische Probleme,
Nachr. K{\"o}nigl. Ges. Wiss. G{\"o}ttingen, Math.-phys. Kl. (1900), 253-297;
reprinted in Arch. Math. Phys. {\bf (3)1} (1901), 44-63, 213-237; also in
Hilbert's Gesammelte Abhandlungen, Vol. III, Chelsea, 1965 (originally in 1935), 
pp.290-329.  
\bibitem{Ign02}J. Ignatavi{\v c}i{\=u}t{\.e},
Discrete universality of the Lerch zeta-function,
in ``Abstracts, 8th Vilnius Conf. on Probab. Theory'', TEV, 2002, pp.116-117.
\bibitem{JanLau13}K. Janulis and A. Laurin{\v c}ikas,
Joint universality of Dirichlet $L$-functions and Hurwitz zeta-functions,
Ann. Univ. Sci. Budapest., Sect. Comp. {\bf 39} (2013), 203-214.
\bibitem{JanLauMacSia12}K. Janulis, A. Laurin{\v c}ikas, R. Macaitien{\.e} and
D. {\v S}iau{\v c}i{\=u}nas,
Joint universality of Dirichlet $L$-functions and periodic Hurwitz zeta-functions,
Math. Modell. Anal. {\bf 17} (2012), 673-685.
\bibitem{JavLau06}A. Javtokas and A. Laurin{\v c}ikas,
On the periodic Hurwitz zeta-function,
Hardy-Ramanujan J. {\bf 29} (2006), 18-36.
\bibitem{JavLau06b}A. Javtokas and A. Laurin{\v c}ikas,
Universality of the periodic Hurwitz zeta-function,
Integr. Transf. Spec. Funct. {\bf 17} (2006), 711-722.
\bibitem{JavLau08}A. Javtokas and A. Laurin{\v c}ikas,
A joint universality theorem for periodic Hurwitz zeta-functions,
Bull. Austral. Math. Soc. {\bf 78} (2008), 13-33.
\bibitem{JesTor45}B. Jessen and H. Tornehave,
Mean motions and zeros of almost periodic functions,
Acta Math. {\bf 77} (1945), 137-279.
\bibitem{KacLau98}A. Ka{\v c}{\.e}nas and A. Laurin{\v c}ikas,
On Dirichlet series related to certain cusp forms, 
Liet. Mat. Rink. {\bf 38} (1998), 82-97 (in Russian); Lith. Math. J. {\bf 38}
(1998), 64-76.
\bibitem{Kac02}R. Ka{\v c}inskait{\.e},
A discrete universality theorem for the Matsumoto zeta-function,
Liet. Mat. Rink. {\bf 42}, Spec. Issue (2002), 55-58.
\bibitem{Kac11}R. Ka{\v c}inskait{\.e},
Joint discrete universality of periodic zeta-functions,
Integr. Transf. Spec. Funct. {\bf 22} (2011), 593-601.
\bibitem{Kac12}R. Ka{\v c}inskait{\.e}, 
Limit theorems for zeta-functions --- with application in universality,
{\v S}iauliai Math. Semin. {\bf 7(15)} (2012), 19-40.
\bibitem{KacJavSia08}R. Ka{\v c}inskait{\.e}, A. Javtokas and
D. {\v S}iau{\v c}i{\=u}nas,
On discrete universality of the periodic zeta-function,
{\v S}iauliai Math. Semin. {\bf 3(11)} (2008), 141-152.
\bibitem{KacLau11}R. Ka{\v c}inskait{\.e} and A. Laurin{\v c}ikas,
The joint distribution of periodic zeta-functions,
Studia Sci. Math. Hung. {\bf 48} (2011), 257-279.
\bibitem{KacSteSiaSle04}R. Ka{\v c}inskait{\.e}, J. Steuding, 
D. {\v S}iau{\v c}i{\=u}nas and R. {\v S}le{\v z}evi{\v c}ien{\.e},
On polynomials in Dirichlet series,
Fiz. Mat. Fak. Moksl. Sem. Darbai, {\v S}iauliai Univ. {\bf 7} (2004), 26-32.
\bibitem{Kacz09}J. Kaczorowski,
Some remarks on the universality of periodic $L$-functions,
in ``New Directions in Value-Distribution Theory of Zeta and $L$-Functions'',
R.\& J. Steuding (eds.), Shaker Verlag, 2009, pp.113-120.
\bibitem{KacKul07}J. Kaczorowski and M. Kulas,
On the non-trivial zeros off the critical line for $L$-functions from the
extended Selberg class,
Monatsh. Math. {\bf 150} (2007), 217-232.
\bibitem{KaczLauSte06}J. Kaczorowski, A. Laurin{\v c}ikas and J. Steuding,
On the value distribution of shifts of universal Dirichlet series,
Monatsh. Math. {\bf 147} (2006), 309-317.
\bibitem{KaczPer99}J. Kaczorowski and A. Perelli,
The Selberg class: a survey,
in ``Number Theory in Progress'', Proc. Intern. Conf. on Number Theory in Honor
of the 60th Birthday of A. Schinzel at Zakopane, Vol. 2, Elementary and Analytic
Number Theory, K. Gy{\"o}ry et al. (eds.), Walter de Gruyter, 1999, pp.953-992. 
\bibitem{KarVor92}A. A. Karatsuba and S. M. Voronin,
The Riemann Zeta-Function, 
Walter de Gruyter, 1992.
\bibitem{KarPanpre}E. Karikovas and {\L}. Pa{\'n}kowski,
Self-approximation of Hurwitz zeta-functions with rational parameter,
Lith. Math. J. {\bf 54} (2014), 74-81.
\bibitem{KiLee11}H. Ki and Y. Lee,
On the zeros of degree one $L$-functions from the extended Selberg class,
Acta Arith. {\bf 149} (2011), 23-36.
\bibitem{KorPocSia13}D. Korsakien{\.e}, V. Pocevi{\v c}ien{\.e} and
D. {\v S}iau{\v c}i{\=u}nas,
On universality of periodic zeta-functions,
{\v S}iauliai Math. Semin. {\bf 8(16)} (2013), 131-141.
\bibitem{Lau79}A. Laurin{\v c}ikas,
Distribution des valeurs de certaines s{\'e}ries de Dirichlet,
C. R. Acad. Sci. Paris {\bf 289} (1979), 43-45.
\bibitem{Lau79b}A. Laurin{\v c}ikas,
Sur les s{\'e}ries de Dirichlet et les polyn{\^o}mes trigonom{\'e}triques,
S{\'e}m. Th{\'e}or. Nombr., Univ. de Bordeaux I, {\'E}xpos{\'e} no.24, 1979.
\bibitem{Lau82}A. Laurin{\v c}ikas,
Distribution of values of generating Dirichlet series of multiplicative functions,
Liet. Mat. Rink. {\bf 22} (1982), 101-111 (in Russian); Lith. Math. J. {\bf 22}
(1982), 56-63.
\bibitem{Lau83}A. Laurin{\v c}ikas,
The universality theorem, 
Liet. Mat. Rink. {\bf 23} (1983), 53-62 (in Russian); Lith. Math. J. {\bf 23}
(1983), 283-289.
\bibitem{Lau84}A. Laurin{\v c}ikas,
The universality theorem II,
Liet. Mat. Rink. {\bf 24} (1984), 113-121 (in Russian); Lith. Math. J. {\bf 24}
(1984), 143-149.
\bibitem{Lau85}A. Laurin{\v c}ikas,
Zeros of the derivative of the Riemann zeta-function,
Liet. Mat. Rink. {\bf 25} (1985), 111-118 (in Russian); Lith. Math. J. {\bf 25}
(1985), 255-260.
\bibitem{Lau86}A. Laurin{\v c}ikas,
Zeros of linear combinations of Dirichlet series,
Liet. Mat. Rink. {\bf 26} (1986), 468-477 (in Russian); Lith. Math. J. {\bf 26}
(1986), 244-251.
\bibitem{Lau95}A. Laurin{\v c}ikas,
On the universality of the Riemann zeta-function,
Liet. Mat. Rink. {\bf 35} (1995), 502-507 (in Russian); Lith. Math. J. {\bf 35}
(1995), 399-402.
\bibitem{Lau96}A. Laurin{\v c}ikas,
Limit Theorems for the Riemann Zeta-function,
Kluwer, 1996.
\bibitem{Lau97}A. Laurin{\v c}ikas,
The universality of the Lerch zeta-function,
Liet. Mat. Rink. {\bf 37} (1997), 367-375 (in Russian); Lith. Math. J. {\bf 37}
(1997), 275-280.
\bibitem{Lau98}A. Laurin{\v c}ikas,
On the Matsumoto zeta-function,
Acta Arith. {\bf 84} (1998), 1-16.
\bibitem{Lau98b}A. Laurin{\v c}ikas,
On the zeros of linear combinations of the Matsumoto zeta-functions,
Liet. Mat. Rink. {\bf 38} (1998), 185-204 (in Russian); Lith. Math. J. {\bf 38}
(1998), 144-159.
\bibitem{Lau98c}A. Laurin{\v c}ikas,
On the Lerch zeta-function with rational parameters,
Liet. Mat. Rink. {\bf 38} (1998), 113-124 (in Russian); Lith. Math. J. {\bf 38}
(1998), 89-97.
\bibitem{Lau00}A. Laurin{\v c}ikas,
On the effectivization of the universality theorem for the Lerch zeta-function,
Liet. Mat. Rink. {\bf 40} (2000), 172-178 (in Russian); Lith. Math. J. {\bf 40}
(2000), 135-139.
\bibitem{Lau01}A. Laurin{\v c}ikas,
The universality of Dirichlet series attached to finite Abelian groups,
in ``Number Theory'', M. Jutila and T. Mets{\"a}nkyl{\"a} (eds.), Walter de Gruyter,
2001, 179-192.
\bibitem{Lau03}A. Laurin{\v c}ikas,
The universality of zeta-functions,
Acta Appl. Math. {\bf 78} (2003), 251-271.
\bibitem{Lau03b}A. Laurin{\v c}ikas,
The joint universality for general Dirichlet series,
Ann. Univ. Sci. Budapest., Sect. Comp. {\bf 22} (2003), 235-251.
\bibitem{Lau05}A. Laurin{\v c}ikas,
Joint universality of general Dirichlet series,
Izv. Ross. Akad. Nauk Ser. Mat. {\bf 69} (2005), 133-144 (in Russian);
Izv. Math. {\bf 69} (2005), 131-142.
\bibitem{Lau05b}A. Laurin{\v c}ikas,
On the derivatives of zeta-functions of certain cusp forms,
Glasgow Math. J. {\bf 47} (2005), 87-96.
\bibitem{Lau05c}A. Laurin{\v c}ikas,
On the derivatives of zeta-functions of certain cusp forms II,
Glasgow Math. J. {\bf 47} (2005), 505-516.
\bibitem{Lau06}A. Laurin{\v c}ikas,
The joint universality for periodic Hurwitz zeta-functions,
Analysis {\bf 26} (2006), 419-428.
\bibitem{Lau07}A. Laurin{\v c}ikas,
Voronin-type theorem for periodic Hurwitz zeta-functions,
Mat. Sb. {\bf 198} (2007), 91-102 (in Russian);
Sb. Math. {\bf 198} (2007), 231-242.
\bibitem{Lau08}A. Laurin{\v c}ikas,
Joint universality for periodic Hurwitz zeta-functions,
Izv. Ross. Akad. Nauk Ser. Mat. {\bf 72} (2008), 121-140 (in Russian);
Izv. Math. {\bf 72} (2008), 741-760.
\bibitem{Lau08b}A. Laurin{\v c}ikas,
The joint universality of Hurwitz zeta-functions,
{\v S}iauliai Math. Semin. {\bf 3(11)} (2008), 169-187.
\bibitem{Lau10}A. Laurin{\v c}ikas,
On the joint universality of Lerch zeta functions,
Mat. Zametki {\bf 88} (2010), 428-437 (in Russian);
Math. Notes {\bf 88} (2010), 386-394.
\bibitem{Lau10b}A. Laurin{\v c}ikas,
Joint universality of zeta-functions with periodic coefficients,
Izv. Ross. Akad. Nauk Ser. Mat. {\bf 74} (2010), 79-102 (in Russian);
Izv. Math. {\bf 74} (2010), 515-539.
\bibitem{Lau10c}A. Laurin{\v c}ikas,
Universality of the Riemann zeta-function,
J. Number Theory {\bf 130} (2010), 2323-2331.
\bibitem{Lau12}A. Laurin{\v c}ikas,
Joint universality of Hurwitz zeta-functions,
Bull. Austral. Math. Soc. {\bf 86} (2012), 232-243.
\bibitem{Lau12b}A. Laurin{\v c}ikas,
Universality of composite functions of periodic zeta functions,
Mat. Sb. {\bf 203} (2012), 105-120 (in Rissian);
Sb. Math. {\bf 203} (2012), 1631-1646.
\bibitem{Lau12c}A. Laurin{\v c}ikas,
Universality of composite functions,
in ``Functions in Number Theory and Their Probabilistic Aspects'', K. Matsumoto
et al. (eds.), RIMS K{\^o}ky{\^u}roku Bessatsu {\bf B34}, RIMS, 2012, pp.191-204.
\bibitem{Lau13}A. Laurin{\v c}ikas,
On the universality of the Hurwitz zeta-function,
Intern. J. Number Theory {\bf 9} (2013), 155-165.
\bibitem{Lau13b}A. Laurin{\v c}ikas,
Universality results for the Riemann zeta-function,
Moscow J. Combin. Number Theory {\bf 3} (2013), 237-256.
\bibitem{LauJNT}A. Laurin{\v c}ikas,
A discrete universality theorem for the Hurwitz zeta-function,
J. Number Theory, to appear.
\bibitem{LauMS}A. Laurin{\v c}ikas,
Joint discrete universality of Hurwitz zeta-functions,
preprint (in Russian).
\bibitem{LauGar02}A. Laurin{\v c}ikas and R. Garunk{\v s}tis,
The Lerch Zeta-function,
Kluwer, 2002.
\bibitem{LauMac09}A. Laurin{\v c}ikas and R. Macaitien{\.e},
On the joint universality of periodic zeta functions,
Mat. Zametki {\bf 85} (2009), 54-64 (in Russian);
Math. Notes {\bf 85} (2009), 51-60.
\bibitem{LauMac09b}A. Laurin{\v c}ikas and R. Macaitien{\.e},
The discrete universality of the periodic Hurwitz zeta function,
Integr. Transf. Spec. Funct. {\bf 20} (2009), 673-686.
\bibitem{LauMac12}A. Laurin{\v c}ikas and R. Macaitien{\.e},
On the universality of zeta-functions of certain cusp forms,
in ``Analytic and Probabilistic Methods in Number Theory'' (Kubilius Memorial
Volume), A. Laurin{\v c}ikas et al. (eds.), TEV, 2012,
pp.173-183.
\bibitem{LauMac13}A. Laurin{\v c}ikas and R. Macaitien{\.e},
Joint universality of the Riemann zeta-function and Lerch zeta-functions,
Nonlinear Anal. Modell. Control {\bf 18} (2013), 314-326.
\bibitem{LauMacSia07}A. Laurin{\v c}ikas, R. Macaitien{\.e} and 
D. {\v S}iau{\v c}i{\=u}nas,
The joint universality for periodic zeta-functions,
Chebyshevski{\u\i} Sb. {\bf 8} (2007), 162-174.
\bibitem{LauMacSia09}A. Laurin{\v c}ikas, R. Macaitien{\.e} and
D. {\v S}iau{\v c}i{\=u}nas,
On discrete universality of the periodic zeta-function II,
in ``New Directions in Value-Distribution Theory of Zeta and $L$-Functions'',
R.\& J. Steuding (eds.), Shaker Verlag, 2009, pp.149-159.
\bibitem{LauMacSia11}A. Laurin{\v c}ikas, R. Macaitien{\.e} and
D. {\v S}iau{\v c}i{\=u}nas,
Joint universality for zeta-functions of different types,
Chebyshevski{\u\i} Sb. {\bf 12} (2011), 192-203.
\bibitem{LauMat00}A. Laurin{\v c}ikas and K. Matsumoto,
The joint universality and the functional independence for Lerch zeta-functions,
Nagoya Math. J. {\bf 157} (2000), 211-227.
\bibitem{LauMat01}A. Laurin{\v c}ikas and K. Matsumoto,
The universality of zeta-functions attached to certain cusp forms,
Acta Arith. {\bf 98} (2001), 345-359.
\bibitem{LauMat02}A. Laurin{\v c}ikas and K. Matsumoto,
The joint universality of zeta-functions attached to certain cusp forms,
Fiz. Mat. Fak. Moksl. Sem. Darbai, {\v S}iauliai Univ. {\bf 5} (2002), 58-75.
\bibitem{LauMat04}A. Laurin{\v c}ikas and K. Matsumoto,
The joint universality of twisted automorphic $L$-functions,
J. Math. Soc. Japan {\bf 56} (2004), 923-939.
\bibitem{LauMat06}A. Laurin{\v c}ikas and K. Matsumoto,
Joint value-distribution theorems on Lerch zeta-functions II,
Liet. Mat. Rink. {\bf 46} (2006), 332-350; Lith. Math. J. {\bf 46} (2006), 271-286.
\bibitem{LauMat07}A. Laurin{\v c}ikas and K. Matsumoto,
Joint value-distribution theorems on Lerch zeta-functions III,
in ``Analytic and Probabilistic Methods in Number Theory'' (Proc. 4th Palanga
Conf.), A. Laurin{\v c}ikas and E. Manstavi{\v c}ius (eds.), TEV, 2007,
pp.87-98.
\bibitem{LauMatSte03}A. Laurin{\v c}ikas, K. Matsumoto and J. Steuding,
The universality of $L$-functions associated with new forms,
Izv. Ross. Akad. Nauk Ser. Mat. {\bf 67} (2003), 83-98 (in Russian);
Izv. Math. {\bf 67} (2003), 77-90.
\bibitem{LauMatSte05}A. Laurin{\v c}ikas, K. Matsumoto and J. Steuding,
Discrete universality of $L$-functions for new forms,
Mat. Zametki {\bf 78} (2005), 595-603 (in Russian);
Math. Notes {\bf 78} (2005), 551-558.
\bibitem{LauMatSte13}A. Laurin{\v c}ikas, K. Matsumoto and J. Steuding,
Universality of some functions related to zeta-functions of certain cusp forms,
Osaka J. Math. {\bf 50} (2013), 1021-1037.
\bibitem{LauMatSte13b}A. Laurin{\v c}ikas, K. Matsumoto and J. Steuding,
On hybrid universality of certain composite functions involving Dirichlet
$L$-functions,
Ann. Univ. Sci. Budapest., Sect. Comp. {\bf 41} (2013), 85-96.
\bibitem{LauRas12}A. Laurin{\v c}ikas and J. Ra{\v s}yt{\.e},
Generalizations of a discrete universality theorem for Hurwitz zeta-functions,
Lith. Math. J. {\bf 52} (2012), 172-180.
\bibitem{LauSchSte03}A. Laurin{\v c}ikas, W. Schwarz and J. Steuding,
The universality of general Dirichlet series,
Analysis {\bf 23} (2003), 13-26.
\bibitem{LauSia06}A. Laurin{\v c}ikas and D. {\v S}iau{\v c}i{\=u}nas,
Remarks on the universality of the periodic zeta function,
Mat. Zametki {\bf 80} (2006), 561-568 (in Russian);
Math. Notes {\bf 80} (2006), 532-538.
\bibitem{LauSia12}A. Laurin{\v c}ikas and D. {\v S}iau{\v c}i{\=u}nas,
A mixed joint universality theorem for zeta-functions III,
in ``Analytic and Probabilistic Methods in Number Theory'' (Kubilius Memorial
Volume), A. Laurin{\v c}ikas et al. (eds.), TEV, 2012,
pp.185-195.
\bibitem{LauSke09}A. Laurin{\v c}ikas and S. Skerstonait{\.e},
A joint universality theorem for periodic Hurwitz zeta-functions 
II\footnote{This ``II'' is probably to be deleted.},
Lith. Math. J. {\bf 49} (2009), 287-296.
\bibitem{LauSke09b}A. Laurin{\v c}ikas and S. Skerstonait{\.e},
Joint universality for periodic Hurwitz zeta-functions II,
in ``New Directions in Value-Distribution Theory of Zeta and $L$-Functions'',
R.\& J. Steuding (eds.), Shaker Verlag, 2009, pp.161-169.
\bibitem{LauSle02}A. Laurin{\v c}ikas and R. {\v S}le{\v z}evi{\v c}ien{\.e},
The universality of zeta-functions with multiplicative coefficients,
Integr. Transf. Spec. Funct. {\bf 13} (2002), 243-257.
\bibitem{Lee12}Y. Lee,
The universality theorem for Hecke $L$-functions,
Math. Z. {\bf 271} (2012), 893-909.
\bibitem{Lee12b}Y. Lee,
Zeros of partial zeta functions off the critical line,
in ``Functions in Number Theory and Their Probabilistic Aspects'', K. Matsumoto
et al. (eds.), RIMS K{\^o}ky{\^u}roku Bessatsu {\bf B34}, RIMS, 2012, pp.205-216.
\bibitem{LiWu07}H. Li and J. Wu,
The universality of symmetric power $L$-functions and their Rankin-Selberg
$L$-functions,
J. Math. Soc. Japan {\bf 59} (2007), 371-392.
\bibitem{Mac06}R. Macaitien{\.e},
A discrete universality theorem for general Dirichlet series,
Analysis {\bf 26} (2006), 373-381.
\bibitem{Mac12}R. Macaitien{\.e},
On joint universality for the zeta-functions of newforms and periodic Hurwitz
zeta-functions,
in ``Functions in Number Theory and Their Probabilistic Aspects'', K. Matsumoto
et al. (eds.), RIMS K{\^o}ky{\^u}roku Bessatsu {\bf B34}, RIMS, 2012, pp.217-233.
\bibitem{Mar35}J. Marcinkiewicz,
Sur les nombres d{\'e}riv{\'e}s, 
Fund. Math. {\bf 24} (1935), 305-308.
\bibitem{Mat90}K. Matsumoto, 
Value-distribution of zeta-functions, 
in ``Analytic Number Theory'', Proc. Japanese-French Sympos., K. Nagasaka and
E. Fouvry (eds.), Lecture Notes in Math. {\bf 1434}, Springer, 1990, pp.178-187. 
\bibitem{Mat01}K. Matsumoto,
The mean values and the universality of Rankin-Selberg $L$-functions,
in ``Number Theory'', M. Jutila and T. Mets{\"a}nkyl{\"a} (eds.), Walter de Gruyter,
2001, 201-221.
\bibitem{Mat02}K. Matsumoto,
Some problems on mean values and the universality of zeta and multiple
zeta-functions,
in ``Analytic and Probabilistic Methods in Number Theory''  (Proc. 3rd Palanga
Conf.), A. Dubickas et al (eds.), TEV, 2002, pp.195-199.
\bibitem{Mat04}K. Matsumoto,
Probabilistic value-distribution theory of zeta-functions,
S{\=u}gaku {\bf 53} (2001), 279-296 (in Japanese);
Sugaku Expositions {\bf 17} (2004), 51-71.
\bibitem{Mat06}K. Matsumoto,
An introduction to the value-distribution theory of zeta-functions,
{\v S}iauliai Math. Semin. {\bf 1(9)} (2006), 61-83.
\bibitem{Mau07}J.-L. Mauclaire, 
Almost periodicity and Dirichlet series,
in ``Analytic and Probabilistic Methods in Number Theory'' (Proc. 4th Palanga
Conf.), A. Laurin{\v c}ikas and E. Manstavi{\v c}ius (eds.), TEV, 2007,
pp.109-142.
\bibitem{Mau09}J.-L. Mauclaire,
On some Dirichlet series,
in ``New Directions in Value-Distribution Theory of Zeta and $L$-Functions'',
R.\& J. Steuding (eds.), Shaker Verlag, 2009, pp.171-248.
\bibitem{Mau13}J.-L. Mauclaire,
Simple remarks on some Dirichlet series,
Ann. Univ. Sci. Budapest., Sect. Comp. {\bf 41} (2013), 159-172.
\bibitem{Mer51}S. N. Mergelyan,
On the representation of functions by series of polynomials on closed sets,
Dokl. Akad. Nauk SSSR {\bf 78} (1951), 405-408 (in Russian);
Amer. Math. Soc. Transl. Ser.1, {\bf 3}, Series and Approximation, Amer. Math. Soc.,
1962, pp.287-293.
\bibitem{Mer52}S. N. Mergelyan,
Uniform approximation to functions of complex variable,
Usp. Mat. Nauk {\bf 7} (1952), 31-122 (in Russian);
Amer. Math. Soc. Transl. Ser.1, {\bf 3}, Series and Approximation, Amer. Math. Soc.,
1962, pp.294-391.
\bibitem{Mey11}T. Meyrath,
On the universality of derived functions of the Riemann zeta-function,
J. Approx. Theory {\bf 163} (2011), 1419-1426.
\bibitem{Mis01}H. Mishou,
The universality theorem for $L$-functions associated with ideal class characters,
Acta Arith. {\bf 98} (2001), 395-410.
\bibitem{Mis03}H. Mishou,
The universality theorem for Hecke $L$-functions,
Acta Arith. {\bf 110} (2003), 45-71.
\bibitem{Mis03b}H. Mishou,
On the value distribution of Hecke $L$-functions associated with
gr{\"o}ssencharacters, 
S{\^u}rikaiseki Kenky{\^u}sho K{\^o}ky{\^u}roku {\bf 1324} (2003), 174-182
(in Japanese).
\bibitem{Mis04}H. Mishou,
The value distribution of Hecke $L$-functions in the gr{\"o}ssencharacter aspect,
Arch. Math. {\bf 82} (2004), 301-310.
\bibitem{Mis05}H. Mishou,
The universality theorem for Hecke $L$-functions in the $(m,t)$ aspect,
Tokyo J. Math. {\bf 28} (2005), 139-153.
\bibitem{Mis07}H. Mishou,
The joint value-distribution of the Riemann zeta function and Hurwitz zeta functions,
Liet. Mat. Rink. {\bf 47} (2007), 62-80; Lith. Math. J. {\bf 47} (2007), 32-47.
\bibitem{Mis09}H. Mishou,
The universality theorem for cubic $L$-functions,
in ``New Directions in Value-Distribution Theory of Zeta and $L$-Functions'',
R.\& J. Steuding (eds.), Shaker Verlag, 2009, pp.265-274.
\bibitem{Mis11}H. Mishou,
The universality theorem for class group $L$-functions,
Acta Arith. {\bf 147} (2011), 115-128.
\bibitem{Mis11b}H. Mishou,
The joint universality theorem for a pair of Hurwitz zeta functions,
J. Number Theory {\bf 131} (2011), 2352-2367.
\bibitem{Mis12}H. Mishou,
On joint universality for derivatives of the Riemann zeta function and automorphic
$L$-functions,
in ``Functions in Number Theory and Their Probabilistic Aspects'', K. Matsumoto
et al. (eds.), RIMS K{\^o}ky{\^u}roku Bessatsu {\bf B34}, RIMS, 2012, pp.235-246.
\bibitem{Mis13}H. Mishou,
Joint value distribution for zeta functions in disjoint strips,
Monatsh. Math. {\bf 169} (2013), 219-247.
\bibitem{MisJMSJ}H. Mishou,
Functional distribution for a collection of Lerch zeta functions,
J. Math. Soc. Japan, to appear.
\bibitem{MisMZ}H. Mishou,
Joint universality theorems for pairs of automorphic zeta functions,
Math. Z., to appear.
\bibitem{MisKoy02}H. Mishou and S. Koyama,
Universality of Hecke $L$-functions in the Grossencharacter-aspect,
Proc. Japan Acad. {\bf 78A} (2002), 63-67.
\bibitem{MisNag06}H. Mishou and H. Nagoshi,
Functional distribution of $L(s,\chi_d)$ with real characters and denseness of
quadratic class numbers,
Trans. Amer. Math. Soc. {\bf 358} (2006), 4343-4366.
\bibitem{MisNag06b}H. Mishou and H. Nagoshi,
The universality of quadratic $L$-series for prime discriminants,
Acta Arith. {\bf 123} (2006), 143-161.
\bibitem{MisNag06c}H. Mishou and H. Nagoshi,
Equivalents of the Riemann hypothesis,
Arch. Math. {\bf 86} (2006), 419-424.
\bibitem{MisNag12}H. Mishou and H. Nagoshi,
On class numbers of quadratic fields with prime discriminant and character sums,
Kyushu J. Math. {\bf 66} (2012), 21-34.
\bibitem{Mon77}H. L. Montgomery,
Extreme values of the Riemann zeta-function,
Comment. Math. Helv. {\bf 52} (1977), 511-518.
\bibitem{Nag05}H. Nagoshi,
On the universality for $L$-functions attached to Maass forms,
Analysis {\bf 25} (2005), 1-22.
\bibitem{Nag07}H. Nagoshi,
The universality of $L$-functions attached to Maass forms,
in ``Probability and Number Theory --- Kanazawa 2005'', S. Akiyama et al. (eds.),
Adv. Stud. Pure Math. {\bf 49}, Math. Soc. Japan, 2007, pp.289-306.
\bibitem{Nag09}H. Nagoshi,
Value-distribution of Rankin-Selberg $L$-functions,
in ``New Directions in Value-Distribution Theory of Zeta and $L$-Functions", 
R.\& J. Steuding (eds.), Shaker Verlag, 2009, pp.275-287.
\bibitem{NagSte10}H. Nagoshi and J. Steuding,
Universality for $L$-functions in the Selberg class,
Lith. Math. J. {\bf 50} (2010), 293-311.
\bibitem{Nak07}T. Nakamura,
Applications of inversion formulas to the joint $t$-universality of Lerch zeta
functions,
J. Number Theory {\bf 123} (2007), 1-9.
\bibitem{Nak07b}T. Nakamura,
The existence and the non-existence of joint $t$-universality for Lerch zeta
functions,
J. Number Theory {\bf 125} (2007), 424-441.
\bibitem{Nak08}T. Nakamura,
Joint value approximation and joint universality for several types of zeta
functions,
Acta Arith. {\bf 134} (2008), 67-82.
\bibitem{Nak09}T. Nakamura,
Zeros and the universality for the Euler-Zagier-Hurwitz type of multiple
zeta-functions,
Bull. London Math. Soc. {\bf 41} (2009), 691-700.
\bibitem{Nak09b}T. Nakamura,
The joint universality and the generalized strong recurrence for Dirichlet
$L$-functions,
Acta Arith. {\bf 138} (2009), 357-262.
\bibitem{Nak10}T. Nakamura,
The generalized strong recurrence for non-zero rational parameters,
Arch. Math. {\bf 95} (2010), 549-555.
\bibitem{Nak11}T. Nakamura,
The universality for linear combinations of Lerch zeta functions and the
Tornheim-Hurwitz type of double zeta functions,
Monatsh. Math. {\bf 162} (2011), 167-178.
\bibitem{Nak11b}T. Nakamura,
Some topics related to universality of $L$-functions with an Euler product,
Analysis {\bf 31} (2011), 31-41.
\bibitem{NakPan11}T. Nakamura and {\L}. Pa{\'n}kowski,
Applications of hybrid universality to multivariable zeta-functions,
J. Number Theory {\bf 131} (2011), 2151-2161.
\bibitem{NakPan12}T. Nakamura and {\L}. Pa{\'n}kowski,
Erratum to ``The generalized strong recurrence for non-zero rational parameters'',
Arch. Math. {\bf 99} (2012), 43-47.
\bibitem{NakPan12b}T. Nakamura and {\L}. Pa{\'n}kowski,
On universality for linear combinations of $L$-functions,
Monatsh. Math. {\bf 165} (2012), 433-446.
\bibitem{NakPan13}T. Nakamura and {\L}. Pa{\'n}kowski,
Self-approximation for the Riemann zeta function,
Bull. Austral. Math. Soc. {\bf 87} (2013), 452-461.
\bibitem{NakPan13b}T. Nakamura and {\L}. Pa{\'n}kowski,
On zeros and $c$-values of Epstein zeta-functions,
{\v S}iauliai Math. Semin. {\bf 8(16)} (2013), 181-195.
\bibitem{NakPanpre}T. Nakamura and {\L}. Pa{\'n}kowski,
On complex zeros off the critical line for non-monomial polynomial of zeta-functions,
preprint.
\bibitem{Pal}J. P{\'a}l,
Zwei kleine Bemerkungen,
T{\^o}hoku Math. J. {\bf 6} (1914/15), 42-43. 
\bibitem{Pan09}{\L}. Pa{\'n}kowski,
Some remarks on the generalized strong recurrence for $L$-functions,
in ``New Directions in Value-Distribution Theory of Zeta and $L$-Functions'',
R.\& J. Steuding (eds.), Shaker Verlag, 2009, pp.305-315.
\bibitem{Pan10}{\L}. Pa{\'n}kowski,
Hybrid universality theorem for Dirichlet $L$-functions,
Acta Arith. {\bf 141} (2010), 59-72.
\bibitem{Pan13}{\L}. Pa{\'n}kowski,
Hybrid universality theorem for $L$-functions without Euler product,
Integr. Transf. Spec. Funct. {\bf 24} (2013), 39-49.
\bibitem{Pec73}D. V. Pecherski{\u\i},
On rearrangements of terms in functional series,
Dokl. Akad. Nauk SSSR {\bf 209} (1973), 1285-1287 (in Russian);
Soviet Math. Dokl. {\bf 14} (1973), 633-636.
\bibitem{PocSia}V. Pocevi{\v c}ien{\.e} and D. {\v S}iau{\v c}i{\=u}nas,
A mixed joint universality theorem for zeta-functions II,
Math. Modell. Anal. {\bf 19} (2014), 52-65.
\bibitem{Ram92}K. Ramachandra,
On Riemann zeta-function and allied questions,
in ``Journ{\'e}es Arithm{\'e}tiques de Gen{\`e}ve'', D. F. Coray and
Y.-F. S. P{\'e}termann (eds.), Ast{\'e}risque {\bf 209}, Soc. Math. France,
1992, pp.57-72. 
\bibitem{Rei77}A. Reich,
Universalle Werteverteilung von Eulerprodukten,
Nachr. Akad. Wiss. G{\"o}ttingen II Math.-Phys. Kl., Nr.1 (1977), 1-17.
\bibitem{Rei80}A. Reich,
Werteverteilung von Zetafunktionen, 
Arch. Math. {\bf 34} (1980), 440-451.
\bibitem{Rei82}A. Reich,
Zur Universit{\"a}t und Hypertranszendenz der Dedekindschen Zetafunktion,
Abh. Braunschweig. Wiss. Ges. {\bf 33} (1982), 197-203.
\bibitem{SanSte06}J. Sander and J. Steuding,
Joint universality for sums and products of Dirichlet $L$-functions,
Analysis {\bf 26} (2006), 295-312.
\bibitem{SchSteSte07}W. Schwarz, R. Steuding and J. Steuding,
Universality for Euler products, and related arithmetical functions,
in ``Analytic and Probabilistic Methods in Number Theory'' (Proc. 4th Palanga
Conf.), A. Laurin{\v c}ikas and E. Manstavi{\v c}ius (eds.), TEV, 2007, 
pp.163-189. 
\bibitem{Sel92}A. Selberg,
Old and new conjectures and results about a class of Dirichlet series,
in ``Proceedings of the Amalfi Conference on Analytic Number Theory'',
E. Bombieri et al. (eds.), Univ. di Salerno, 1992, pp.367-385;
also in Selberg's Collected Papers, Vol. II, Springer, 1991, pp.47-63. 
\bibitem{Sle02}
R. {\v S}le{\v z}evi{\v c}ien{\.e}\footnote{R. {\v S}le{\v z}evi{\v c}ien{\.e}
= R. Steuding},
The joint universality for twists of Dirichlet series with multiplicative
coefficients by characters,
in ``Analytic and Probabilistic Methods in Number Theory''  (Proc. 3rd Palanga
Conf.), A. Dubickas et al (eds.), TEV, 2002, pp.303-319.
\bibitem{SriSteSte13}T. Srichan, R. Steuding and J. Steuding,
Does a random walker meet universality?
{\v S}iauliai Math. Semin. {\bf 8(16)} (2013), 249-259. 
\bibitem{Ste02}J. Steuding,
The world of $p$-adic numbers and $p$-adic functions,
Fiz. Mat. Fak. Moksl. Sem. Darbai, {\v S}iauliai Univ. {\bf 5} (2002), 90-107.
\bibitem{Ste03}J. Steuding,
On the universality for functions in the Selberg class,
in ``Proc. Session in Analytic Number Theory and Diophantine Equations'',
D. R. Heath-Brown and B. Z. Moroz (eds.), Bonner Math. Schriften {\bf 360},
Bonn, 2003, n.28, 22pp.
\bibitem{Ste03b}J. Steuding,
Value-distribution of $L$-functions and allied zeta-functions --- with an
emphasis on aspects of universality, Habilitationsschrift, Frankfurt, 
Johann Wolfgang Goethe Universit{\"a}t, 2003.
\bibitem{Ste03c}J. Steuding,
Upper bounds for the density of universality,
Acta Arith. {\bf 107} (2003), 195-202.
\bibitem{Ste05}J. Steuding,
Upper bounds for the density of universality II,
Acta Math. Univ. Ostrav. {\bf 13} (2005), 73-82.
\bibitem{Ste07}J. Steuding,
Value-distribution of $L$-functions,
Lecture Notes in Math. {\bf 1877}, Springer, 2007.
\bibitem{Ste13}J. Steuding,
Ergodic universality theorems for Riemann's zeta-function and other $L$-functions,
J. Th{\'e}or. Nombr. Bordeaux {\bf 25} (2013), 471-476.
\bibitem{RSte06}R. Steuding,
Universality for generalized Euler products,
Analysis {\bf 26} (2006), 337-345.
\bibitem{Vor72}S. M. Voronin,
On the distribution of nonzero values of the Riemann $\zeta$-function,
Trudy Mat. Inst. Steklov. {\bf 128} (1972), 131-150 (in Russian);
Proc. Steklov Inst. Math. {\bf 128} (1972), 153-175.
\bibitem{Vor75}S. M. Voronin,
Theorem on the "universality' of the Riemann zeta-function, 
Izv. Akad. Nauk SSSR Ser. Mat. {\bf 39} (1975), 475-486 (in Russian);
Math. USSR Izv. {\bf 9} (1975), 443-453.
\bibitem{Vor75b}S. M. Voronin,
On the functional independence of Dirichlet $L$-functions,
Acta Arith. {\bf 27} (1975), 443-453 (in Russian).
\bibitem{Vor77}S. M. Voronin,
Analytic properties of Dirichlet generating functions of arithmetic
objects,
Thesis, Moscow, Steklov Math. Institute, 1977 (in Russian).
\bibitem{Vor79}S. M. Voronin,
Analytic properties of Dirichlet generating functions of arithmetic
objects,
Mat. Zametki {\bf 24} (1978), 879-884 (in Russian); 
Math. Notes {\bf 24} (1979), 966-969.
\bibitem{Vor85}S. M. Voronin,
On the distribution of zeros of some Dirichlet series,
Trudy Mat. Inst. Steklov. {\bf 163} (1984), 74-77 (in Russian);
Proc. Steklov Inst. Math. (1985), Issue 4, 89-92.
\bibitem{Vor88}S. M. Voronin,
On $\Omega$-theorems in the theory of the Riemann zeta-function,
Izv. Akad. Nauk SSSR Ser. Mat. {\bf 52} (1988), 424-436 (in Russian);
Math. USSR Izv. {\bf 32} (1989), 429-442.


\end{thebibliography}
\end{document}